\documentclass[10pt]{amsproc}
\usepackage{pfl}
\title[An extended Demazure product]{An extended Demazure product on integer permutations via min-plus matrix multiplication}
\pflauthorinfo
\date{\today}

\usepackage[skins]{tcolorbox}
\usepackage{fontawesome5}
\usetikzlibrary{decorations.pathreplacing}

\newif\ifshowleanrefs
\showleanrefstrue

\usepackage{enumitem}
\usepackage{multicol}

\definecolor{leanlistingbg}{gray}{0.96}
\definecolor{leanlistingrule}{gray}{0.70}

\newcommand{\doclinkIcon}{%
  \textsuperscript{\fontsize{4pt}{4pt}\selectfont\faIcon{external-link-alt}}%
}
\ExplSyntaxOn
\cs_new_protected:Npn \doclink_text:n #1
  {
    \seq_set_split:Nnn \l_tmpa_seq { . } { #1 }
    \texttt{\tl_to_str:e { \seq_item:Nn \l_tmpa_seq { -1 } }}
  }
\cs_new_protected:Npn \doclink_href:nn #1 #2
  {
    \tl_if_blank:nTF { #1 }
      {
        \seq_set_split:Nnn \l_tmpa_seq { . } { #2 }
        \int_compare:nNnTF { \seq_count:N \l_tmpa_seq } > { 1 }
          {
            \seq_pop_right:NN \l_tmpa_seq \l_tmpa_tl
            \tl_set:Nx \l_tmpb_tl { \seq_use:Nn \l_tmpa_seq { . } }
          }
          { \tl_set:Nn \l_tmpb_tl { #2 } }
      }
      { \tl_set:Nn \l_tmpb_tl { #1 } }
    \href{https://npflueger.com/demazure/docs/Demazure/\tl_use:N \l_tmpb_tl.html\##2}%
         {\doclinkText{#2}\,\doclinkIcon}%
  }
\NewDocumentCommand{\doclinkText}{m}{\doclink_text:n { #1 }}
\NewDocumentCommand{\doclink}{O{}m}{\doclink_href:nn { #1 } { #2 }}
\ExplSyntaxOff
\newtcbox{\leanboxframe}{%
  on line,
  enhanced,
  colback=leanlistingbg,
  colframe=leanlistingrule,
  boxrule=0.35pt,
  arc=2pt,
  left=3pt,
  right=3pt,
  top=1pt,
  bottom=1pt,
  boxsep=0pt
}
\newcommand{\leanbox}[1]{%
  \ifshowleanrefs\leanboxframe{\tiny #1}\fi
}
\NewDocumentCommand{\leanlink}{O{}m}{%
  \leanbox{\doclink[#1]{#2}}%
}
\newcommand{\leanEquation}[3]{%
  \begin{equation*}
  \refstepcounter{equation}\label{#1}%
  \makebox[0pt][l]{\normalfont\normalcolor(\theequation)}%
  \makebox[\displaywidth][c]{$\displaystyle #2$}%
  \makebox[0pt][r]{\leanbox{#3}}%
  \end{equation*}%
}
\newcommand{\leanTaggedEquation}[4]{%
  \begin{equation*}
  \renewcommand{\theequation}{#4}%
  \refstepcounter{equation}\label{#1}%
  \makebox[0pt][l]{\normalfont\normalcolor(\theequation)}%
  \makebox[\displaywidth][c]{$\displaystyle #2$}%
  \makebox[0pt][r]{\leanbox{#3}}%
  \end{equation*}%
}
\makeatletter
\newcommand{\eqnum}[1]{%
  \refstepcounter{equation}\ltx@label{#1}%
  \normalfont\normalcolor(\theequation)%
}
\makeatother

\DeclareMathOperator{\ess}{Ess}
\newcommand{\asp}{\operatorname{ASP}}
\newcommand{\Inv}{\operatorname{Inv}}
\newcommand{\slip}{\operatorname{SF}}
\newcommand{\sa}{s_\alpha}
\newcommand{\sbe}{s_\beta}
\newcommand{\sai}{s_{\alpha^{-1}}}
\newcommand{\sbei}{s_{\beta^{-1}}}
\newcommand{\ca}{\chi_\alpha}
\newcommand{\cb}{\chi_\beta}
\newcommand{\sab}{s_{\alpha \star \beta}}
\newcommand{\sabl}{s_{\alpha \resL \beta}}

\newcommand{\si}[1]{s_{\iota_{#1}}}
\newcommand{\ts}{\widetilde{S}}
\newcommand{\mab}{M_{\alpha \star \beta}}
\newcommand{\mtab}{M_{\alpha \resL \beta}}
\newcommand{\resR}{\triangleright}
\newcommand{\resL}{\triangleleft}
\newcommand{\reduced}[1]{\mathbin{{#1}_{\mathrm{red}}}}
\newcommand{\starr}{\reduced{\star}}
\newcommand{\resLr}{\reduced{\resL}}

\newcommand{\tsig}{\widetilde{\sigma}}
\newcommand{\lechi}{\leq_\chi}
\newcommand{\Out}{\operatorname{Out}}
\newcommand{\In}{\operatorname{In}}

\theoremstyle{definition} 
\newtheorem{warning}[thm]{Warning}

\begin{document}

\begin{abstract}
Coxeter groups possess an associative operation, called variously the Demazure, greedy, or $0$-Hecke product. For symmetric groups, this product has an amusing formulation, due to Tiskin, as matrix multiplication in the min-plus (tropical) semiring of two matrices associated to the permutations. We prove that this min-plus formulation extends to furnish a Demazure product on a much larger group of integer permutations, consisting of all permutations that change the sign of finitely many integers. We prove several alternative descriptions of this product and some useful properties. These results were developed in service of Brill--Noether theory of algebraic and tropical curves; the connection is surveyed in an appendix. The main theorems of this paper have been fully formalized in Lean 4.
\end{abstract}

\maketitle


\section{Introduction}

The symmetric group $S_d$, and more generally any Coxeter group, possesses a useful associative operation $\star$, variously called the \emph{Demazure}, \emph{$0$-Hecke}, or \emph{greedy} product. If $\alpha \in S_d$ and $\sigma_n$ is the simple transposition of $n$ and $n+1$, then
\begin{equation}
\label{eq:demazureBasic}
\alpha \star \sigma_n = \begin{cases}
\alpha \sigma_n & \mbox{ if } \alpha(n) < \alpha(n+1),\\
\alpha & \mbox{ otherwise.}
\end{cases}
\end{equation}
Since $\star$ is associative and every $\beta \in S_d$ factors into simple transpositions, Equation \eqref{eq:demazureBasic} is enough to compute any Demazure product in $S_d$. A similar formulation is valid in any Coxeter group.

The Demazure product on $S_d$ has an equivalent formulation that is transparently independent of choice of factorization. To state it, first define, for any permutation $\alpha$, a counting function
$$\sa(a,b) = \# \{ \ell \geq b:\ \alpha(\ell) < a \}.$$
See Figure \ref{fig:sa} for an illustration. The Demazure product on $S_d$ then has the following characterization, 

\begin{equation}
\label{eq:demazureMin}
\sab(a,b) = \min_{1 \leq \ell \leq d+1} \Big[ \sa(a,\ell) + \sbe(\ell,b) \Big], \hspace{1cm} \mbox{for all } 1 \leq a,b \leq d+1.
\end{equation}

\begin{figure}
\begin{tikzpicture}[scale=0.2]
\fill[lightgray] (4.5,3.5) rectangle (12,-3);
\fill[lightgray] (4.5,3.5) rectangle (-3,12);

\draw[->] (-3,0) -- (12,0);
\draw[->] (0,-3) -- (0,12);

\node [ circle, inner sep=1pt] at (-3,-3) {};
\node [fill=black, circle, inner sep=1pt] at (-2,-2) {};
\node [fill=black, circle, inner sep=1pt] at (-1,-1) {};
\node [fill=black, circle, inner sep=1pt] at (0,0) {};
\node [fill=black, circle, inner sep=1pt] at (1,5) {};
\node [fill=black, circle, inner sep=1pt] at (2,6) {};
\node [fill=black, circle, inner sep=1pt] at (3,2) {};
\node [fill=black, circle, inner sep=1pt] at (4,8) {};
\node [fill=black, circle, inner sep=1pt] at (5,3) {};
\node [fill=black, circle, inner sep=1pt] at (6,9) {};
\node [fill=black, circle, inner sep=1pt] at (7,7) {};
\node [fill=black, circle, inner sep=1pt] at (8,4) {};
\node [fill=black, circle, inner sep=1pt] at (9,1) {};
\node [fill=black, circle, inner sep=1pt] at (10,10) {};
\node [fill=black, circle, inner sep=1pt] at (11,11) {};
\node [fill=black, circle, inner sep=1pt] at (12,12) {};

\draw[decoration={brace,mirror,raise=2pt},decorate]
  (12,-3) -- (12,3.5) node[midway,right] {$\ \sa(4,5) = 2$};
  
\draw[decoration={brace,mirror,raise=2pt},decorate] (-3,12) -- (-3,3.5) node[midway, left] {$\sai(5,4) = 3\ $};

\foreach \x in {1,...,9} {
\draw (\x,0.25) -- (\x,-0.25) node[below] {\tiny$\x$};
}
\foreach \y in {1,...,9} {
\draw (0.25,\y) -- (-0.25,\y) node[left] {\tiny$\y$};
}

\end{tikzpicture}

\caption{
Illustration of the permutation $\alpha \in S_9$ given in one-line notation by $5\ 6\ 2\ 8\ 3\ 9\ 7\ 4\ 1$, and the functions $\sa$ and $\sai$. By convention, the permutation is extended by the identity to a permutation $\ZZ \to \ZZ$.
}
\label{fig:sa}
\end{figure}
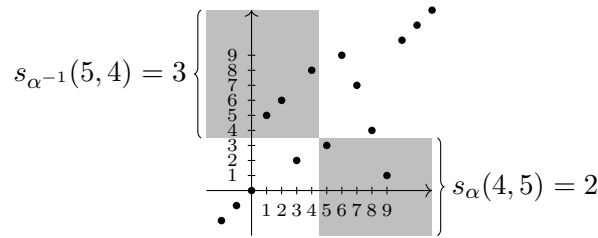

This min-plus formulation of $\star$ on $S_d$ was found in a different form by Tiskin \cite{tiskinMonge}, which also gave an efficient algorithm for its computation. Tiskin's point of view has also been generalized to some other infinite contexts including biwords on $\{1,\cdots,n\}$ \cite{erickson25}, permutons on $[0,1]$ \cite{defantPermutons}, and affine permutations \cite{gztStringComparison}. Other algorithmic extensions include \cite{ggkMonge}.

Functions like $\sa$ (perhaps with reversed inequalities or other minor modifications) are sometimes called ``rank functions;'' they provide a convenient way to state the strong Bruhat order on $S_d$ \cite[Theorem 2.1.5]{bjornerBrenti} and are used to define Schubert varieties in flag varieties \cite[\S 10]{fultonTableaux}. 
If we view the functions $\sa, \sbe$ as $(d+1) \times (d+1)$ matrices, Equation \eqref{eq:demazureMin} says that $\sab$ is obtained by matrix multiplication in the min-plus (tropical) semiring. It has a nice geometric interpretation in terms of relative position of flags; see Appendix \ref{app:geometric}. 

We will follow the convention that elements of $S_d$ are extended to permutations of $\ZZ$ that fix all $n \not\in [1,d]$. This does not affect the values $\sa(a,b)$ for $1 \leq a,b \leq d+1$, but it does allow an alternative form of Equation \eqref{eq:demazureMin} that does not depend on the choice of $d$:

\begin{equation}
\label{eq:demazureMinZ}
\sab(a,b) = \min_{\ell \in \ZZ} \Big[ \sa(a,\ell) + \sbe(\ell,b) \Big], \hspace{1cm} \mbox{for all } a,b \in \ZZ.
\end{equation}

This equation generalizes nicely to larger groups than $S_d$. For example, the \emph{affine symmetric groups $\widetilde{S}_k$} (see Example \ref{eg:extAffine} or Section \ref{ssec:extAffine} for a definition) are Coxeter groups embedded as subgroups of the permutations of $\ZZ$, and Equation \eqref{eq:demazureMinZ} holds, with $\sigma_n$ replaced by different generators, in $\widetilde{S}_k$ as well. A simple proof is in \cite{pflBriefTropBN} and is sketched in \cite{cpRR}; this also follows from this paper's results.

The present paper uses the min-plus description of $\star$ to generalize it to a much wider class of permutations. 
This generalization makes it possible to combine elements of $S_d$ and elements of $\widetilde{S}_k$ for different values of $k$. Though this may appear at first glance to be an idle curiosity, such ``mixed modulus'' Demazure products occur in the theory of divisors on graphs and algebraic curves. Indeed, the impetus to write this paper was originally to provide a self-contained reference for various facts needed for work on tropical and algebraic curves.
A brief discussion of these applications may be found in Appendix \ref{app:curvesGraphs}.

The group considered herein is the group of \emph{almost-sign-preserving permutations,} which are the largest class of permutations $\alpha: \ZZ \to \ZZ$ for which $\sa$ and $\sai$ both take only finite values.
This group is uncountable, and defining $\star$ in terms of a generating set is not feasible. We prove the existence of an associative product $\star$ in this context, defined by the min-plus formulation (Theorem \ref{thm:starExists}), and demonstrate that it shares several nice aspects of the Demazure product on $S_d$, including a characterization as a ``greedy product'' in Bruhat order (Theorem \ref{thm:starGreedy}) and a ``reduction theorem'' (Theorem \ref{thm:reduce}) for simplifying inequalities of the form $\alpha \star \beta \geq \gamma$.

We also study a slightly broader context than $\asp$, namely a set $\slip$ of functions which we call \emph{slipface functions} $s: \ZZ^2 \to \NN$. Then $\asp$ is embedded in $\slip$ by $\alpha \mapsto \sa$; the image of this embedding consists of \emph{submodular} slipface functions. The non-submodular slipface functions occur in tropical geometry, so there is some utility in this further generalization. The generalization from $\asp$ to $\slip$ is analogous to the generalization from permutations matrices to alternating sign matrices; see Remark \ref{rem:asm}.

It is plausible that many of this paper's results can be deduced from existing theory of Coxeter groups, via limit arguments. This would certainly be interesting to carry out, and indeed earlier drafts of this paper follow this approach in some parts (for example, Theorem \ref{thm:starGreedy} may be proved by multiplying simple transpositions one at a time and applying Zorn's lemma). However, this paper avoids this approach, and instead frames everything in terms of the functions $\sa$ and Equation \eqref{eq:demazureMinZ}. This reveals some intriguing combinatorics hiding in the min-plus matrix multiplication expression, and also means that this paper is entirely self-contained.

The main theorems of this paper, and many but not yet all of the auxiliary results, have been formally verified in the Lean 4 proof assistant. 
The source code is available at the following repository, which also contains instructions for building and verifying the code. 
\begin{center}
\url{https://github.com/npflueger/demazure}
\end{center}
\ifshowleanrefs%
This paper also contains links to documentation of the formalizations of specific results, which are formatted like this, and shown near the right margin: \leanlink{AspPerm.AspPerm}.%
\fi
This project is still in development, with the objective eventually of using it as a tool with which to formalize aspects of combinatorial Brill--Noether theory. I welcome collaborators who wish to contribute or build upon the code.

\subsection{Conventions and key definitions}

The symbol $\NN$ denotes the set of \emph{nonnegative} numbers. The symbol $\delta$ is used for an indicator function, equal to $1$ if the statement within is true, and $0$ otherwise. A \emph{permutation} always refers to a bijective function $\ZZ \to \ZZ$. Whenever we refer to a permutation of $\{1, \cdots, d\}$, or any other finite set, we implicitly extend this to a permutation of all of $\ZZ$ that fixes every other integer; so $S_d$ is always implicitly understood to be embedded in the full permutation group of $\ZZ$.

Call a permutation $\alpha: \ZZ \to \ZZ$ \emph{almost-sign-preserving} if there are only finitely many $n$ such that exactly one of $n, \alpha(n)$ is negative. We denote the group of almost-sign-preserving permutations by $\asp$.
 For any $\alpha \in \asp$, define a function $\sa : \ZZ^2 \to \NN$ by
\begin{equation}
\label{eq:sa}
\sa(a,b) = \# \{ n \geq b: \alpha(n) < a \}.
\end{equation}
This function will be called (somewhat whimsically) the \emph{slipface function} of $\alpha$; see Section \ref{sec:aspPrelim}.

The \emph{shift} of a permutation $\alpha \in \asp$ is the number 
$$\ca = \sa(0,0) - \sai(0,0).$$
The importance of this number will become clear as we develop our results; some intuition about it is provided by observing that elements of $S_d$ have shift $0$, and more generally a permutation $\alpha$ with finitely many inversions satisfies $\alpha(n) = n - \ca$ for all but finitely many $n \in \ZZ$ (see Proposition \ref{prop:reconstruction} for a precise general form of this fact).

The \emph{Bruhat order} on $\asp$ is the partial order $\leq$, where 
$$\alpha \leq \beta \quad \mbox{ means that }\quad  \sa(a,b) \leq \sbe(a,b) \mbox{ for all } a,b \in \ZZ.$$
In principle this definition involves infinitely many inequalities, but in practice it often suffices to check only those $(a,b)$ in the \emph{essential set}; this is examined in Section \ref{sec:ess}.

Restricted to $S_d$, this is the usual ``strong'' Bruhat order \cite[Theorem 2.1.5]{bjornerBrenti}, which can also be deduced from Theorem \ref{thm:reduceSeveral}.
We discuss an analog in $\asp$ of the ``weak'' Bruhat orders in Section \ref{sec:reducedWeak}. 
We will also use the following shorthand: $$\alpha \lechi \beta \;\quad \text{means that} \quad \alpha \leq \beta \text{ and } \chi_\alpha = \chi_\beta.$$

\subsection{Main results}
The following existence theorem combines Theorems \ref{thm:starExists1} and \ref{thm:alphaStarSigma}.

\begin{customthm}{A}
\label{thm:starExists}
There is an associative operation $\star$ on $\asp$, characterized by \hfill\leanlink[Submodular]{AspPerm.star_assoc}
\leanEquation{eq:sab}{%
\sab(a,b) = \min_{\ell \in \ZZ} \Big[ \sa(a,\ell) + \sbe(\ell,b) \Big]}{\doclink[Submodular]{AspPerm.star_sf_isleast}}
for all $\alpha, \beta \in \asp$ and $a,b \in \ZZ$.
We call $\star$ the \emph{Demazure product} on $\asp$.
For all $\alpha \in \asp$ and $n \in \ZZ$, $\alpha \star \sigma_n$ satisfies Equation \eqref{eq:demazureBasic}, so $\star$ extends the standard Demazure product on $S_d$.
\hfill\leanlink{Transpositions.star_simple}
\end{customthm}

The $S_d$ case has a geometric interpretation in terms of the relative position of three flags; see Appendix \ref{app:geometric}.

The Demazure product coincides with the ordinary product in an important special case.
We will see (Lemma \ref{lem:reducedStar}) that $\alpha \star \beta \geq \alpha \beta$, and that equality holds if and only if $\alpha$ and $\beta^{-1}$ have no inversions in common. In this situation, we will say later that $\alpha \star \beta = \alpha \beta$ is a \emph{reduced product} and write $\alpha \starr \beta$. I encourage the reader to attempt to prove this before reading on, as it provides useful intuition about Equation \eqref{eq:sab}.

We also prove a second characterization of this operation. It generalizes \cite[Lemma 1]{heSubalgebra}, \cite[Lemma 3.1(e)]{buchMihalcea} from $S_d$ to $\asp$, and is proved in Section \ref{sec:mainTheorems}.

\begin{customthm}{B}
\label{thm:starGreedy}
For all $\alpha, \beta \in \asp$, the following maxima in Bruhat order all exist, and equal $\alpha \star \beta$.
\begin{flalign*}
\eqnum{eq:starGreedy}
&& \alpha \star \beta &= \max \left\{ \alpha_1 \beta_1:\ \alpha_1 \leq \alpha \text{ and } \beta_1 \leq \beta \right\} 
& \leanlink{Reduction.starGreedy}
\\
\eqnum{eq:starGreedyBeta}
&& &= \max \left\{ \alpha \beta_1: \beta_1 \lechi \beta \right\} 
& \leanlink{Reduction.starGreedy_beta}
\\
\eqnum{eq:starGreedyAlpha}
&& &= \max \left\{ \alpha_1 \beta: \alpha_1 \lechi \alpha \right\}.
& \leanlink{Reduction.starGreedy_alpha}
\end{flalign*}
\end{customthm}

One way to understand Equation \eqref{eq:demazureBasic} is that for $\beta \in S_d$, factoring $\beta$ into adjacent transpositions allows you to find the $\beta_1$ in Equation \eqref{eq:starGreedyBeta} by a greedy algorithm.

The last main objective of this paper is a ``reduction theorem'' for inequalities of the form $\alpha \star \beta \geq \gamma$. We will prove that such an inequality implies an equation $\alpha_1 \star \beta_1 = \gamma$. Furthermore, we will want these $\alpha_1, \beta_1$ to satisfy $\alpha_1 \beta_1 = \alpha_1 \star \beta_1$, and we wish to show that if $\alpha, \beta, \gamma$ are from a suitable subgroup of $\asp$, then we may choose $\alpha_1, \beta_1$ from the same subgroup. Several examples of such subgroups are discussed in Section \ref{sec:subgroups}, including symmetric and affine symmetric groups. To state our result, we first require an auxiliary operation that will prove to be a useful accomplice to $\star$. This operation was identified by He in \cite[Lemma 1.4]{heMinimal} for Coxeter groups.

\begin{thm}
\label{thm:resL}
For all $\alpha, \beta \in \asp$, the following Bruhat-minimum exists:
\leanEquation{eq:resLMin}{%
\alpha \resL \beta^{-1} = \min \{ \gamma \in \asp:\ \gamma \star \beta \geq \alpha \}.%
}{\doclink[Submodular]{AspPerm.lres_eq_min}}
The resulting operation $\resL$ is called the \emph{left residual}. It is characterized by the equation
\leanEquation{eq:sabl}{%
s_{\alpha \resL \beta}(a,b) = \max_{\ell \in \ZZ} \Big[ \sa(a,\ell) - \sbei(b,\ell) \Big].
}{\doclink[Submodular]{AspPerm.lres_sf_isgreatest}}
For all $\alpha \in \asp$ and $n \in \ZZ$, $\alpha \resL \sigma_n = \begin{cases}
\alpha \sigma_n & \mbox{ if } \alpha(n) > \alpha(n+1).\\
\alpha & \mbox{ otherwise.}
\end{cases}$%
\hfill\leanlink{Transpositions.residual_simple}
\end{thm}

Most of this theorem is proved in Theorem \ref{thm:resLExists}, except the statement about $\sigma_n$, which is proved in Theorem \ref{thm:alphaStarSigma}. The operation $\resL$ also has a characterization analogous to Theorem \ref{thm:starGreedy}, stated in Theorem \ref{thm:resLStingy}, which is closer to He's definition in \cite{heMinimal}.

There is also a dual operation to $\resL$, which we denote $\resR$ and call the \emph{right residual}. It is characterized by
\begin{equation}
\label{eq:resRMin}
\alpha^{-1} \resR \beta = \min \{ \gamma \in \asp: \alpha \star \gamma \geq \beta \}.
\end{equation}
It suffices to focus on $\resL$, since we will prove that $\alpha \resR \beta = (\beta^{-1} \resL \alpha^{-1})^{-1}$ (Theorem \ref{thm:resLExists}).

We mention the case $\beta = \sigma_n$ in the statement above to support the following intuition: where $\star$ is a greedy multiplication, $\resL$ is a ``stingy multiplication;'' $\beta$ attempts to make $\alpha$ as Bruhat-small as possible using part of itself. This is made precise in Theorem \ref{thm:resLStingy}, in analogy to Theorem \ref{thm:starGreedy}.

Our reduction theorem for inequalities $\alpha \star \beta \geq \gamma$ is the following. It is proved in Section \ref{sec:mainTheorems}, along with a generalization to products of three or more permutations (Theorem \ref{thm:reduceSeveral}).

\begin{customthm}{C}
\label{thm:reduce}
Let $G \leq \asp$ be a subgroup that is closed under $\resL$. For all $\alpha, \beta, \gamma \in G$ such that $\ca + \cb = \chi_\gamma$, $\alpha \star \beta \geq \gamma$ if and only if there exist $\alpha_1, \beta_1 \in G$ such that $\alpha_1 \lechi \alpha,\ \beta_1 \lechi \beta$, and $\alpha_1 \star \beta_1 = \alpha_1 \beta_1 = \gamma$.

More specifically, for all $\alpha, \beta, \gamma \in \asp$, if $\alpha \star \beta \geq \gamma$, then $\alpha_1 = \gamma \resL \beta^{-1}$ and $\beta_1 = \alpha_1^{-1} \resR \gamma$ satisfy $\alpha_1 \star \beta_1 = \alpha_1 \beta_1 = \gamma$, 
$\alpha_1 \lechi \alpha$, and $\beta_1 \lechi \beta$.
\hfill\leanlink{Reduction.reduce_witness}
\end{customthm}

In addition to Theorems \ref{thm:starExists}, \ref{thm:starGreedy}, and \ref{thm:reduce}, we also examine some other properties of Bruhat order and $\star$ on certain subgroups of $\asp$ that will be useful in our applications; see Sections \ref{sec:ess} and \ref{sec:subgroups}.

\subsection{Background on Demazure products}
\label{ssec:background}

We use the name \emph{Demazure product} in reference to \cite{demazure}, in which a collection of operators $L_\alpha$ are defined such that, in our notation, $L_\alpha L_\beta = L_{\alpha \star \beta}$ (see \S 5.6 of that paper). The same operation occurred around the same time in \cite{bggSchubert}. This point of view, realizing the Demazure product in terms of composition of operators, occurs in many other contexts, e.g. \cite[Definition 3.1]{knutsonMiller}, which considers operators $\overline{\partial}_i$, $1 \leq i < n$, on a polynomial ring $R[x_1, \cdots, x_n]$ defined by 
$$\overline{\partial_i}(f) =  \frac{x_{i+1} f - x_i (s_i \cdot f)}{x_{i+1}-x_i},$$
and considers the \emph{Demazure algebra} generated by these operators. These difference operators are related, but not identical to, the difference operators used to define Schubert polynomials, e.g. as in \cite[p. 165]{fultonTableaux}. I do not know if there is a useful way to relate the min-plus formulation of $\star$ to this operator formulation, or whether the Demazure product on $\asp$ can be understood in a similar way.

The other principal point of view on $\star$ is from the Hecke algebra, in which it arises (up to sign) by setting the parameter $q$ to $0$, hence called the $0$-Hecke product. See \cite[\S 7]{humphreysReflection} or \cite[\S 6]{bjornerBrenti}.

Applications of the Demazure product to algebraic geometry typically relate to the geometry of Schubert varieties. Another interesting example of such applications is  \cite{buchMihalcea}, which shows that a \emph{curve neighborhood} of a Schubert variety results in another Schubert variety, indexed by the permutation given by the Demazure product.

\subsection{Outline of the paper}

We begin with some preliminary facts about $\asp$ and the functions $\sa$ in Section \ref{sec:aspPrelim}. Section \ref{sec:slipface} introduces the set $\slip$ of \emph{slipfaces}, formulates the Demazure product and residuals at this level of generality, and proves a few basic facts. Section \ref{sec:submodular} identifies the functions $\sa$ as those slipfaces that are \emph{submodular}, thereby constructs $\star$ and $\resL$ on $\asp$ and proves Theorem \ref{thm:starExists}, and analyzes the locations of inversions after these operations. Sections \ref{sec:reducedWeak} and \ref{sec:mainTheorems} leverage the earlier results to relate $\star$ and $\resL$ to reduced products and weak Bruhat orders, to prove Theorems \ref{thm:starGreedy} and \ref{thm:reduce}. 
Section \ref{sec:ess} generalizes a concept called the \emph{essential set} to $\asp$, in order to study Bruhat order.
Section \ref{sec:subgroups} considers a number of subgroups of $\asp$ of special interest, and identifies some aspects of $\star$ and $\resL$ specific to those subgroups. Finally, two appendices follow to provide context. Appendix \ref{app:geometric} provides a concrete geometric interpretation of the ``min-plus matrix multiplication'' formula in the $S_d$ case, in terms of relative positions of flags. Appendix \ref{app:curvesGraphs} briefly summarizes applications to divisor theory on algebraic curves and (metric) graphs.

\section{Preliminaries on almost-sign-preserving permutations}
\label{sec:aspPrelim}

This section collects a few preliminary facts about almost-sign-preserving permutations and their slipface functions, and also attempts to explain my somewhat whimsical choice of the word ``slipface.'' First observe the following characterizations of finite differences of $\sa$. Note that Equation \eqref{eq:Deltasa} shows that the slipface function $\sa$ uniquely determines the permutation $\alpha$.
\begin{flalign*}
\eqnum{eq:b+1}&&
\sa(a,b) - \sa(a,b+1) &= \delta\left( \alpha(b) < a \right)
& \leanlink{AspPerm.b_step}
\\ 
\eqnum{eq:a+1}&&
\sa(a+1,b) - \sa(a,b) &= \delta\left( \alpha^{-1}(a) \geq b \right)
& \leanlink{AspPerm.a_step}
\\
\eqnum{eq:Deltasa}&&
\sa(a+1,b) - \sa(a,b) - \sa(a+1,b+1) + \sa(a,b+1) &= \delta(\alpha(b) = a)
& \leanlink{AspPerm.Delta_eq}
\end{flalign*}

I will now attempt to convince you that the word ``slipface'' does indeed evoke the basic properties of these functions $\sa$.
The function $\sa$ may be usefully visualized as a sequence of functions $f(n) = \sa(n,b)$ for various choices of $b$. For $b$ fixed, this function is a nondecreasing function, beginning at $0$, and eventually stabilizing to have slope $1$. When $b$ increases, Equation \eqref{eq:b+1} shows that part of the function drops, namely all points with $n > \alpha(b)$, so the slope-$1$ line that the graph approaches shifts one unit to the \emph{right}. Some examples are shown in Figure \ref{fig:slipfaces}. In these three examples, we consider three permutations in $S_9$. To me at least, this sequence of graphs evokes a dune of sand which, over time, drops in height as sand slips and falls down the face, and thereby appears to move to the right.

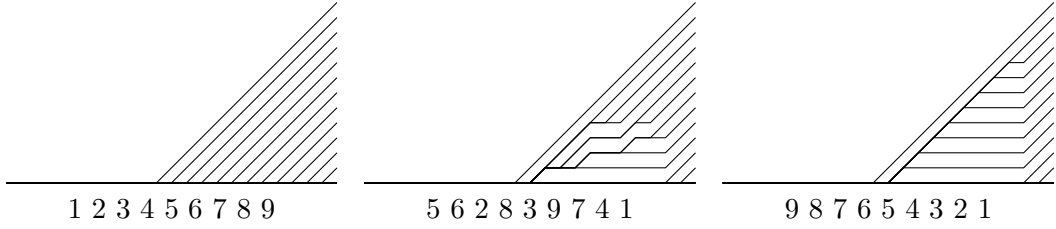
\begin{figure}
\begin{tabular}{ccc}
\begin{tikzpicture}[scale=0.2]
\draw (-10,0) -- (-9,0) -- (-8,0) -- (-7,0) -- (-6,0) -- (-5,0) -- (-4,0) -- (-3,0) -- (-2,0) -- (-1,0) -- (0,0) -- (1,1) -- (2,2) -- (3,3) -- (4,4) -- (5,5) -- (6,6) -- (7,7) -- (8,8) -- (9,9) -- (10,10) -- (11,11) -- (12,12);
\draw (-10,0) -- (-9,0) -- (-8,0) -- (-7,0) -- (-6,0) -- (-5,0) -- (-4,0) -- (-3,0) -- (-2,0) -- (-1,0) -- (0,0) -- (1,0) -- (2,1) -- (3,2) -- (4,3) -- (5,4) -- (6,5) -- (7,6) -- (8,7) -- (9,8) -- (10,9) -- (11,10) -- (12,11);
\draw (-10,0) -- (-9,0) -- (-8,0) -- (-7,0) -- (-6,0) -- (-5,0) -- (-4,0) -- (-3,0) -- (-2,0) -- (-1,0) -- (0,0) -- (1,0) -- (2,0) -- (3,1) -- (4,2) -- (5,3) -- (6,4) -- (7,5) -- (8,6) -- (9,7) -- (10,8) -- (11,9) -- (12,10);
\draw (-10,0) -- (-9,0) -- (-8,0) -- (-7,0) -- (-6,0) -- (-5,0) -- (-4,0) -- (-3,0) -- (-2,0) -- (-1,0) -- (0,0) -- (1,0) -- (2,0) -- (3,0) -- (4,1) -- (5,2) -- (6,3) -- (7,4) -- (8,5) -- (9,6) -- (10,7) -- (11,8) -- (12,9);
\draw (-10,0) -- (-9,0) -- (-8,0) -- (-7,0) -- (-6,0) -- (-5,0) -- (-4,0) -- (-3,0) -- (-2,0) -- (-1,0) -- (0,0) -- (1,0) -- (2,0) -- (3,0) -- (4,0) -- (5,1) -- (6,2) -- (7,3) -- (8,4) -- (9,5) -- (10,6) -- (11,7) -- (12,8);
\draw (-10,0) -- (-9,0) -- (-8,0) -- (-7,0) -- (-6,0) -- (-5,0) -- (-4,0) -- (-3,0) -- (-2,0) -- (-1,0) -- (0,0) -- (1,0) -- (2,0) -- (3,0) -- (4,0) -- (5,0) -- (6,1) -- (7,2) -- (8,3) -- (9,4) -- (10,5) -- (11,6) -- (12,7);
\draw (-10,0) -- (-9,0) -- (-8,0) -- (-7,0) -- (-6,0) -- (-5,0) -- (-4,0) -- (-3,0) -- (-2,0) -- (-1,0) -- (0,0) -- (1,0) -- (2,0) -- (3,0) -- (4,0) -- (5,0) -- (6,0) -- (7,1) -- (8,2) -- (9,3) -- (10,4) -- (11,5) -- (12,6);
\draw (-10,0) -- (-9,0) -- (-8,0) -- (-7,0) -- (-6,0) -- (-5,0) -- (-4,0) -- (-3,0) -- (-2,0) -- (-1,0) -- (0,0) -- (1,0) -- (2,0) -- (3,0) -- (4,0) -- (5,0) -- (6,0) -- (7,0) -- (8,1) -- (9,2) -- (10,3) -- (11,4) -- (12,5);
\draw (-10,0) -- (-9,0) -- (-8,0) -- (-7,0) -- (-6,0) -- (-5,0) -- (-4,0) -- (-3,0) -- (-2,0) -- (-1,0) -- (0,0) -- (1,0) -- (2,0) -- (3,0) -- (4,0) -- (5,0) -- (6,0) -- (7,0) -- (8,0) -- (9,1) -- (10,2) -- (11,3) -- (12,4);
\draw (-10,0) -- (-9,0) -- (-8,0) -- (-7,0) -- (-6,0) -- (-5,0) -- (-4,0) -- (-3,0) -- (-2,0) -- (-1,0) -- (0,0) -- (1,0) -- (2,0) -- (3,0) -- (4,0) -- (5,0) -- (6,0) -- (7,0) -- (8,0) -- (9,0) -- (10,1) -- (11,2) -- (12,3);
\draw (-10,0) -- (-9,0) -- (-8,0) -- (-7,0) -- (-6,0) -- (-5,0) -- (-4,0) -- (-3,0) -- (-2,0) -- (-1,0) -- (0,0) -- (1,0) -- (2,0) -- (3,0) -- (4,0) -- (5,0) -- (6,0) -- (7,0) -- (8,0) -- (9,0) -- (10,0) -- (11,1) -- (12,2);
\draw (-10,0) -- (-9,0) -- (-8,0) -- (-7,0) -- (-6,0) -- (-5,0) -- (-4,0) -- (-3,0) -- (-2,0) -- (-1,0) -- (0,0) -- (1,0) -- (2,0) -- (3,0) -- (4,0) -- (5,0) -- (6,0) -- (7,0) -- (8,0) -- (9,0) -- (10,0) -- (11,0) -- (12,1);
\draw (-10,0) -- (-9,0) -- (-8,0) -- (-7,0) -- (-6,0) -- (-5,0) -- (-4,0) -- (-3,0) -- (-2,0) -- (-1,0) -- (0,0) -- (1,0) -- (2,0) -- (3,0) -- (4,0) -- (5,0) -- (6,0) -- (7,0) -- (8,0) -- (9,0) -- (10,0) -- (11,0) -- (12,0);
\draw (-10,0) -- (-9,0) -- (-8,0) -- (-7,0) -- (-6,0) -- (-5,0) -- (-4,0) -- (-3,0) -- (-2,0) -- (-1,0) -- (0,0) -- (1,0) -- (2,0) -- (3,0) -- (4,0) -- (5,0) -- (6,0) -- (7,0) -- (8,0) -- (9,0) -- (10,0) -- (11,0) -- (12,0);
\draw (-10,0) -- (-9,0) -- (-8,0) -- (-7,0) -- (-6,0) -- (-5,0) -- (-4,0) -- (-3,0) -- (-2,0) -- (-1,0) -- (0,0) -- (1,0) -- (2,0) -- (3,0) -- (4,0) -- (5,0) -- (6,0) -- (7,0) -- (8,0) -- (9,0) -- (10,0) -- (11,0) -- (12,0);
\draw (-10,0) -- (-9,0) -- (-8,0) -- (-7,0) -- (-6,0) -- (-5,0) -- (-4,0) -- (-3,0) -- (-2,0) -- (-1,0) -- (0,0) -- (1,0) -- (2,0) -- (3,0) -- (4,0) -- (5,0) -- (6,0) -- (7,0) -- (8,0) -- (9,0) -- (10,0) -- (11,0) -- (12,0);
\draw (-10,0) -- (-9,0) -- (-8,0) -- (-7,0) -- (-6,0) -- (-5,0) -- (-4,0) -- (-3,0) -- (-2,0) -- (-1,0) -- (0,0) -- (1,0) -- (2,0) -- (3,0) -- (4,0) -- (5,0) -- (6,0) -- (7,0) -- (8,0) -- (9,0) -- (10,0) -- (11,0) -- (12,0);
\draw (-10,0) -- (-9,0) -- (-8,0) -- (-7,0) -- (-6,0) -- (-5,0) -- (-4,0) -- (-3,0) -- (-2,0) -- (-1,0) -- (0,0) -- (1,0) -- (2,0) -- (3,0) -- (4,0) -- (5,0) -- (6,0) -- (7,0) -- (8,0) -- (9,0) -- (10,0) -- (11,0) -- (12,0);
\draw (-10,0) -- (-9,0) -- (-8,0) -- (-7,0) -- (-6,0) -- (-5,0) -- (-4,0) -- (-3,0) -- (-2,0) -- (-1,0) -- (0,0) -- (1,0) -- (2,0) -- (3,0) -- (4,0) -- (5,0) -- (6,0) -- (7,0) -- (8,0) -- (9,0) -- (10,0) -- (11,0) -- (12,0);
\draw (-10,0) -- (-9,0) -- (-8,0) -- (-7,0) -- (-6,0) -- (-5,0) -- (-4,0) -- (-3,0) -- (-2,0) -- (-1,0) -- (0,0) -- (1,0) -- (2,0) -- (3,0) -- (4,0) -- (5,0) -- (6,0) -- (7,0) -- (8,0) -- (9,0) -- (10,0) -- (11,0) -- (12,0);
\draw (-10,0) -- (-9,0) -- (-8,0) -- (-7,0) -- (-6,0) -- (-5,0) -- (-4,0) -- (-3,0) -- (-2,0) -- (-1,0) -- (0,0) -- (1,0) -- (2,0) -- (3,0) -- (4,0) -- (5,0) -- (6,0) -- (7,0) -- (8,0) -- (9,0) -- (10,0) -- (11,0) -- (12,0);
\end{tikzpicture}
&
\begin{tikzpicture}[scale=0.2]
\draw (-10,0) -- (-9,0) -- (-8,0) -- (-7,0) -- (-6,0) -- (-5,0) -- (-4,0) -- (-3,0) -- (-2,0) -- (-1,0) -- (0,0) -- (1,1) -- (2,2) -- (3,3) -- (4,4) -- (5,5) -- (6,6) -- (7,7) -- (8,8) -- (9,9) -- (10,10) -- (11,11) -- (12,12);
\draw (-10,0) -- (-9,0) -- (-8,0) -- (-7,0) -- (-6,0) -- (-5,0) -- (-4,0) -- (-3,0) -- (-2,0) -- (-1,0) -- (0,0) -- (1,0) -- (2,1) -- (3,2) -- (4,3) -- (5,4) -- (6,5) -- (7,6) -- (8,7) -- (9,8) -- (10,9) -- (11,10) -- (12,11);
\draw (-10,0) -- (-9,0) -- (-8,0) -- (-7,0) -- (-6,0) -- (-5,0) -- (-4,0) -- (-3,0) -- (-2,0) -- (-1,0) -- (0,0) -- (1,0) -- (2,1) -- (3,2) -- (4,3) -- (5,4) -- (6,4) -- (7,5) -- (8,6) -- (9,7) -- (10,8) -- (11,9) -- (12,10);
\draw (-10,0) -- (-9,0) -- (-8,0) -- (-7,0) -- (-6,0) -- (-5,0) -- (-4,0) -- (-3,0) -- (-2,0) -- (-1,0) -- (0,0) -- (1,0) -- (2,1) -- (3,2) -- (4,3) -- (5,4) -- (6,4) -- (7,4) -- (8,5) -- (9,6) -- (10,7) -- (11,8) -- (12,9);
\draw (-10,0) -- (-9,0) -- (-8,0) -- (-7,0) -- (-6,0) -- (-5,0) -- (-4,0) -- (-3,0) -- (-2,0) -- (-1,0) -- (0,0) -- (1,0) -- (2,1) -- (3,1) -- (4,2) -- (5,3) -- (6,3) -- (7,3) -- (8,4) -- (9,5) -- (10,6) -- (11,7) -- (12,8);
\draw (-10,0) -- (-9,0) -- (-8,0) -- (-7,0) -- (-6,0) -- (-5,0) -- (-4,0) -- (-3,0) -- (-2,0) -- (-1,0) -- (0,0) -- (1,0) -- (2,1) -- (3,1) -- (4,2) -- (5,3) -- (6,3) -- (7,3) -- (8,4) -- (9,4) -- (10,5) -- (11,6) -- (12,7);
\draw (-10,0) -- (-9,0) -- (-8,0) -- (-7,0) -- (-6,0) -- (-5,0) -- (-4,0) -- (-3,0) -- (-2,0) -- (-1,0) -- (0,0) -- (1,0) -- (2,1) -- (3,1) -- (4,1) -- (5,2) -- (6,2) -- (7,2) -- (8,3) -- (9,3) -- (10,4) -- (11,5) -- (12,6);
\draw (-10,0) -- (-9,0) -- (-8,0) -- (-7,0) -- (-6,0) -- (-5,0) -- (-4,0) -- (-3,0) -- (-2,0) -- (-1,0) -- (0,0) -- (1,0) -- (2,1) -- (3,1) -- (4,1) -- (5,2) -- (6,2) -- (7,2) -- (8,3) -- (9,3) -- (10,3) -- (11,4) -- (12,5);
\draw (-10,0) -- (-9,0) -- (-8,0) -- (-7,0) -- (-6,0) -- (-5,0) -- (-4,0) -- (-3,0) -- (-2,0) -- (-1,0) -- (0,0) -- (1,0) -- (2,1) -- (3,1) -- (4,1) -- (5,2) -- (6,2) -- (7,2) -- (8,2) -- (9,2) -- (10,2) -- (11,3) -- (12,4);
\draw (-10,0) -- (-9,0) -- (-8,0) -- (-7,0) -- (-6,0) -- (-5,0) -- (-4,0) -- (-3,0) -- (-2,0) -- (-1,0) -- (0,0) -- (1,0) -- (2,1) -- (3,1) -- (4,1) -- (5,1) -- (6,1) -- (7,1) -- (8,1) -- (9,1) -- (10,1) -- (11,2) -- (12,3);
\draw (-10,0) -- (-9,0) -- (-8,0) -- (-7,0) -- (-6,0) -- (-5,0) -- (-4,0) -- (-3,0) -- (-2,0) -- (-1,0) -- (0,0) -- (1,0) -- (2,0) -- (3,0) -- (4,0) -- (5,0) -- (6,0) -- (7,0) -- (8,0) -- (9,0) -- (10,0) -- (11,1) -- (12,2);
\draw (-10,0) -- (-9,0) -- (-8,0) -- (-7,0) -- (-6,0) -- (-5,0) -- (-4,0) -- (-3,0) -- (-2,0) -- (-1,0) -- (0,0) -- (1,0) -- (2,0) -- (3,0) -- (4,0) -- (5,0) -- (6,0) -- (7,0) -- (8,0) -- (9,0) -- (10,0) -- (11,0) -- (12,1);
\draw (-10,0) -- (-9,0) -- (-8,0) -- (-7,0) -- (-6,0) -- (-5,0) -- (-4,0) -- (-3,0) -- (-2,0) -- (-1,0) -- (0,0) -- (1,0) -- (2,0) -- (3,0) -- (4,0) -- (5,0) -- (6,0) -- (7,0) -- (8,0) -- (9,0) -- (10,0) -- (11,0) -- (12,0);
\draw (-10,0) -- (-9,0) -- (-8,0) -- (-7,0) -- (-6,0) -- (-5,0) -- (-4,0) -- (-3,0) -- (-2,0) -- (-1,0) -- (0,0) -- (1,0) -- (2,0) -- (3,0) -- (4,0) -- (5,0) -- (6,0) -- (7,0) -- (8,0) -- (9,0) -- (10,0) -- (11,0) -- (12,0);
\draw (-10,0) -- (-9,0) -- (-8,0) -- (-7,0) -- (-6,0) -- (-5,0) -- (-4,0) -- (-3,0) -- (-2,0) -- (-1,0) -- (0,0) -- (1,0) -- (2,0) -- (3,0) -- (4,0) -- (5,0) -- (6,0) -- (7,0) -- (8,0) -- (9,0) -- (10,0) -- (11,0) -- (12,0);
\draw (-10,0) -- (-9,0) -- (-8,0) -- (-7,0) -- (-6,0) -- (-5,0) -- (-4,0) -- (-3,0) -- (-2,0) -- (-1,0) -- (0,0) -- (1,0) -- (2,0) -- (3,0) -- (4,0) -- (5,0) -- (6,0) -- (7,0) -- (8,0) -- (9,0) -- (10,0) -- (11,0) -- (12,0);
\draw (-10,0) -- (-9,0) -- (-8,0) -- (-7,0) -- (-6,0) -- (-5,0) -- (-4,0) -- (-3,0) -- (-2,0) -- (-1,0) -- (0,0) -- (1,0) -- (2,0) -- (3,0) -- (4,0) -- (5,0) -- (6,0) -- (7,0) -- (8,0) -- (9,0) -- (10,0) -- (11,0) -- (12,0);
\draw (-10,0) -- (-9,0) -- (-8,0) -- (-7,0) -- (-6,0) -- (-5,0) -- (-4,0) -- (-3,0) -- (-2,0) -- (-1,0) -- (0,0) -- (1,0) -- (2,0) -- (3,0) -- (4,0) -- (5,0) -- (6,0) -- (7,0) -- (8,0) -- (9,0) -- (10,0) -- (11,0) -- (12,0);
\draw (-10,0) -- (-9,0) -- (-8,0) -- (-7,0) -- (-6,0) -- (-5,0) -- (-4,0) -- (-3,0) -- (-2,0) -- (-1,0) -- (0,0) -- (1,0) -- (2,0) -- (3,0) -- (4,0) -- (5,0) -- (6,0) -- (7,0) -- (8,0) -- (9,0) -- (10,0) -- (11,0) -- (12,0);
\draw (-10,0) -- (-9,0) -- (-8,0) -- (-7,0) -- (-6,0) -- (-5,0) -- (-4,0) -- (-3,0) -- (-2,0) -- (-1,0) -- (0,0) -- (1,0) -- (2,0) -- (3,0) -- (4,0) -- (5,0) -- (6,0) -- (7,0) -- (8,0) -- (9,0) -- (10,0) -- (11,0) -- (12,0);
\draw (-10,0) -- (-9,0) -- (-8,0) -- (-7,0) -- (-6,0) -- (-5,0) -- (-4,0) -- (-3,0) -- (-2,0) -- (-1,0) -- (0,0) -- (1,0) -- (2,0) -- (3,0) -- (4,0) -- (5,0) -- (6,0) -- (7,0) -- (8,0) -- (9,0) -- (10,0) -- (11,0) -- (12,0);
\end{tikzpicture}
&
\begin{tikzpicture}[scale=0.2]
\draw (-10,0) -- (-9,0) -- (-8,0) -- (-7,0) -- (-6,0) -- (-5,0) -- (-4,0) -- (-3,0) -- (-2,0) -- (-1,0) -- (0,0) -- (1,1) -- (2,2) -- (3,3) -- (4,4) -- (5,5) -- (6,6) -- (7,7) -- (8,8) -- (9,9) -- (10,10) -- (11,11) -- (12,12);
\draw (-10,0) -- (-9,0) -- (-8,0) -- (-7,0) -- (-6,0) -- (-5,0) -- (-4,0) -- (-3,0) -- (-2,0) -- (-1,0) -- (0,0) -- (1,0) -- (2,1) -- (3,2) -- (4,3) -- (5,4) -- (6,5) -- (7,6) -- (8,7) -- (9,8) -- (10,9) -- (11,10) -- (12,11);
\draw (-10,0) -- (-9,0) -- (-8,0) -- (-7,0) -- (-6,0) -- (-5,0) -- (-4,0) -- (-3,0) -- (-2,0) -- (-1,0) -- (0,0) -- (1,0) -- (2,1) -- (3,2) -- (4,3) -- (5,4) -- (6,5) -- (7,6) -- (8,7) -- (9,8) -- (10,8) -- (11,9) -- (12,10);
\draw (-10,0) -- (-9,0) -- (-8,0) -- (-7,0) -- (-6,0) -- (-5,0) -- (-4,0) -- (-3,0) -- (-2,0) -- (-1,0) -- (0,0) -- (1,0) -- (2,1) -- (3,2) -- (4,3) -- (5,4) -- (6,5) -- (7,6) -- (8,7) -- (9,7) -- (10,7) -- (11,8) -- (12,9);
\draw (-10,0) -- (-9,0) -- (-8,0) -- (-7,0) -- (-6,0) -- (-5,0) -- (-4,0) -- (-3,0) -- (-2,0) -- (-1,0) -- (0,0) -- (1,0) -- (2,1) -- (3,2) -- (4,3) -- (5,4) -- (6,5) -- (7,6) -- (8,6) -- (9,6) -- (10,6) -- (11,7) -- (12,8);
\draw (-10,0) -- (-9,0) -- (-8,0) -- (-7,0) -- (-6,0) -- (-5,0) -- (-4,0) -- (-3,0) -- (-2,0) -- (-1,0) -- (0,0) -- (1,0) -- (2,1) -- (3,2) -- (4,3) -- (5,4) -- (6,5) -- (7,5) -- (8,5) -- (9,5) -- (10,5) -- (11,6) -- (12,7);
\draw (-10,0) -- (-9,0) -- (-8,0) -- (-7,0) -- (-6,0) -- (-5,0) -- (-4,0) -- (-3,0) -- (-2,0) -- (-1,0) -- (0,0) -- (1,0) -- (2,1) -- (3,2) -- (4,3) -- (5,4) -- (6,4) -- (7,4) -- (8,4) -- (9,4) -- (10,4) -- (11,5) -- (12,6);
\draw (-10,0) -- (-9,0) -- (-8,0) -- (-7,0) -- (-6,0) -- (-5,0) -- (-4,0) -- (-3,0) -- (-2,0) -- (-1,0) -- (0,0) -- (1,0) -- (2,1) -- (3,2) -- (4,3) -- (5,3) -- (6,3) -- (7,3) -- (8,3) -- (9,3) -- (10,3) -- (11,4) -- (12,5);
\draw (-10,0) -- (-9,0) -- (-8,0) -- (-7,0) -- (-6,0) -- (-5,0) -- (-4,0) -- (-3,0) -- (-2,0) -- (-1,0) -- (0,0) -- (1,0) -- (2,1) -- (3,2) -- (4,2) -- (5,2) -- (6,2) -- (7,2) -- (8,2) -- (9,2) -- (10,2) -- (11,3) -- (12,4);
\draw (-10,0) -- (-9,0) -- (-8,0) -- (-7,0) -- (-6,0) -- (-5,0) -- (-4,0) -- (-3,0) -- (-2,0) -- (-1,0) -- (0,0) -- (1,0) -- (2,1) -- (3,1) -- (4,1) -- (5,1) -- (6,1) -- (7,1) -- (8,1) -- (9,1) -- (10,1) -- (11,2) -- (12,3);
\draw (-10,0) -- (-9,0) -- (-8,0) -- (-7,0) -- (-6,0) -- (-5,0) -- (-4,0) -- (-3,0) -- (-2,0) -- (-1,0) -- (0,0) -- (1,0) -- (2,0) -- (3,0) -- (4,0) -- (5,0) -- (6,0) -- (7,0) -- (8,0) -- (9,0) -- (10,0) -- (11,1) -- (12,2);
\draw (-10,0) -- (-9,0) -- (-8,0) -- (-7,0) -- (-6,0) -- (-5,0) -- (-4,0) -- (-3,0) -- (-2,0) -- (-1,0) -- (0,0) -- (1,0) -- (2,0) -- (3,0) -- (4,0) -- (5,0) -- (6,0) -- (7,0) -- (8,0) -- (9,0) -- (10,0) -- (11,0) -- (12,1);
\draw (-10,0) -- (-9,0) -- (-8,0) -- (-7,0) -- (-6,0) -- (-5,0) -- (-4,0) -- (-3,0) -- (-2,0) -- (-1,0) -- (0,0) -- (1,0) -- (2,0) -- (3,0) -- (4,0) -- (5,0) -- (6,0) -- (7,0) -- (8,0) -- (9,0) -- (10,0) -- (11,0) -- (12,0);
\draw (-10,0) -- (-9,0) -- (-8,0) -- (-7,0) -- (-6,0) -- (-5,0) -- (-4,0) -- (-3,0) -- (-2,0) -- (-1,0) -- (0,0) -- (1,0) -- (2,0) -- (3,0) -- (4,0) -- (5,0) -- (6,0) -- (7,0) -- (8,0) -- (9,0) -- (10,0) -- (11,0) -- (12,0);
\draw (-10,0) -- (-9,0) -- (-8,0) -- (-7,0) -- (-6,0) -- (-5,0) -- (-4,0) -- (-3,0) -- (-2,0) -- (-1,0) -- (0,0) -- (1,0) -- (2,0) -- (3,0) -- (4,0) -- (5,0) -- (6,0) -- (7,0) -- (8,0) -- (9,0) -- (10,0) -- (11,0) -- (12,0);
\draw (-10,0) -- (-9,0) -- (-8,0) -- (-7,0) -- (-6,0) -- (-5,0) -- (-4,0) -- (-3,0) -- (-2,0) -- (-1,0) -- (0,0) -- (1,0) -- (2,0) -- (3,0) -- (4,0) -- (5,0) -- (6,0) -- (7,0) -- (8,0) -- (9,0) -- (10,0) -- (11,0) -- (12,0);
\draw (-10,0) -- (-9,0) -- (-8,0) -- (-7,0) -- (-6,0) -- (-5,0) -- (-4,0) -- (-3,0) -- (-2,0) -- (-1,0) -- (0,0) -- (1,0) -- (2,0) -- (3,0) -- (4,0) -- (5,0) -- (6,0) -- (7,0) -- (8,0) -- (9,0) -- (10,0) -- (11,0) -- (12,0);
\draw (-10,0) -- (-9,0) -- (-8,0) -- (-7,0) -- (-6,0) -- (-5,0) -- (-4,0) -- (-3,0) -- (-2,0) -- (-1,0) -- (0,0) -- (1,0) -- (2,0) -- (3,0) -- (4,0) -- (5,0) -- (6,0) -- (7,0) -- (8,0) -- (9,0) -- (10,0) -- (11,0) -- (12,0);
\draw (-10,0) -- (-9,0) -- (-8,0) -- (-7,0) -- (-6,0) -- (-5,0) -- (-4,0) -- (-3,0) -- (-2,0) -- (-1,0) -- (0,0) -- (1,0) -- (2,0) -- (3,0) -- (4,0) -- (5,0) -- (6,0) -- (7,0) -- (8,0) -- (9,0) -- (10,0) -- (11,0) -- (12,0);
\draw (-10,0) -- (-9,0) -- (-8,0) -- (-7,0) -- (-6,0) -- (-5,0) -- (-4,0) -- (-3,0) -- (-2,0) -- (-1,0) -- (0,0) -- (1,0) -- (2,0) -- (3,0) -- (4,0) -- (5,0) -- (6,0) -- (7,0) -- (8,0) -- (9,0) -- (10,0) -- (11,0) -- (12,0);
\draw (-10,0) -- (-9,0) -- (-8,0) -- (-7,0) -- (-6,0) -- (-5,0) -- (-4,0) -- (-3,0) -- (-2,0) -- (-1,0) -- (0,0) -- (1,0) -- (2,0) -- (3,0) -- (4,0) -- (5,0) -- (6,0) -- (7,0) -- (8,0) -- (9,0) -- (10,0) -- (11,0) -- (12,0);
\end{tikzpicture}
\\
1 2 3 4 5 6 7 8 9
&
5 6 2 8 3 9 7 4 1
&
9 8 7 6 5 4 3 2 1
\end{tabular}
\caption{Superimposed plots of the graphs $y = s_\alpha(x,b)$ for $0 \leq b \leq 20$, and three different permutations $\alpha$ of $\{1,\cdots,9\}$. Each permutation is written in one-line notation $\alpha(1)\ \alpha(2)\ \cdots\ \alpha(9)$.
The graph in Figure \ref{fig:sa} may be useful in studying the middle example.
}
\label{fig:slipfaces}
\end{figure}

\subsection{Some properties of the shift $\ca$}
Equations \eqref{eq:b+1} and \eqref{eq:a+1} show that the function $f(a,b) = \sa(a,b) - \sai(b,a) -a + b$ is constant.  Since $\ca = f(0,0)$ by definition, we have the following duality between $\sa$ and $\sai$, reminiscent (not coincidentally!) of the Riemann-Roch formula. 
\leanTaggedEquation{eq:saDuality}
  {\sa(a,b) - \sai(b,a) = \ca + a - b.}
  {\doclink{AspPerm.duality}}
  {\ensuremath{\dagger}}
Equation \eqref{eq:saDuality} shows that $\sa(a,b) \geq \max\{0, \ca + a - b\}$, with equality if and only if either $\sa(a,b) = 0$ or $\sai(b, a) = 0$. For $a$ fixed and $b \ll 0$, or $b$ fixed and $a \gg 0$, $\sai(b,a) = 0$, so $\sa(a,b) = \ca + a - b$. In this sense, the shift $\ca$ governs the asymptotic behavior of $\sa$. We mention a few other immediate consequences of Equation \eqref{eq:saDuality}.

\begin{lemma}
\label{lem:bruhatInverse}
For all $\alpha, \beta \in \asp$, $\alpha \lechi \beta$ if and only if $\alpha^{-1} \lechi \beta^{-1}$.
\hfill\leanlink[Submodular]{AspPerm.le_chi_inv_iff}
\end{lemma}
\begin{proof}
If $\ca = \cb$, then Equation \eqref{eq:saDuality} implies $\sa(a,b) - \sbe(a,b) = \sai(b,a) - \sbei(b,a)$.
\end{proof}

\begin{defn}
\label{defn:shiftPerm}
For any $\chi \in \ZZ$, denote by $\iota_\chi$ the increasing permutation $\iota_\chi(n) = n-\chi$, of shift $\chi$.
\end{defn}

\begin{lemma}
For any $\chi \in \ZZ$ and $\alpha \in \asp$, if $\ca \geq \chi$ then $\alpha \geq \iota_\chi$. In particular,
$$\iota_\chi = \min \{ \alpha \in \asp:\ \ca = \chi \}.$$
\end{lemma}

\begin{proof}
For all $a,b \in \ZZ$, $\sa(a,b) \geq \sai(b,a) + \chi + a - b \geq \max \{ 0, \chi+a-b\} = s_{\iota_\chi}(a,b).$
\end{proof}

\newcommand{\sgn}{\operatorname{sgn}}
If we define $\sgn(n)$ to be $1$ for $n \geq 0$ and $-1$ for $n < 0$, then 
$$\ca = \frac12 \sum_{n \in \ZZ} \Big( \sgn(n) - \sgn(\alpha(n)) \Big).$$
It follows from this description, rearranging a sum and cancelling, that $\chi$ is a homomorphism:
\leanEquation{eq:chiHom}
  {\chi_{\alpha \beta} = \chi_\alpha + \chi_\beta.}
  {\doclink{AspPerm.chi_mul}}
Thus shift sorts $\asp$ into cosets of a homomorphism $\chi: \asp \to \ZZ$. 
This homomorphism has a nice splitting, given by $\chi \mapsto \iota_\chi$.
We will see that $\chi$ is also a monoid homomorphism for the Demazure product: $\chi_{\alpha \star \beta} = \ca + \cb$ (Theorem \ref{thm:starExists1}). Furthermore, $\iota_m \star \alpha = \iota_m \alpha$ and $\alpha \star \iota_m = \alpha \iota_m$ for all $\alpha \in \asp$ (this follows from Lemma \ref{lem:reducedStar}, since $\iota_\chi$ has no inversions). This means that essentially everything we might wish to know about $\star$ on $\asp$ is determined by its restriction to $\asp_0 = \{ \alpha \in \asp: \ca = 0 \}$.

\begin{eg}
\label{eg:extAffine}
For any $k \geq 2$, the group of permutations $\alpha: \ZZ \to \ZZ$ such that $\alpha(n+k) = \alpha(n) + k$ for all $n \in \ZZ$ is called an \emph{extended affine symmetric group.} For such a permutation $\alpha$, sorting $\ZZ$ into cosets modulo $k$ and counting sign changes shows that 
$\ds \chi_\alpha = - \frac1{k} \sum_{n=0}^{k-1} \Big(\alpha(n)-n \Big).$
 One standard definition of the (unextended) affine symmetric group $\widetilde{S}_k$ is that its elements satisfy $\alpha(n+k) = \alpha(n)+k$ for all $n$ and $\ds \sum_{n=0}^{k-1} \Big(\alpha(n)-n\Big) = 0$. So in our terminology, this is the shift-$0$ subgroup of the extended affine symmetric group.
\end{eg}

\subsection{Inversions, reduced products, and the weak orders}
Much of our analysis will be based on careful study of \emph{inversions}.

\begin{defn}
\label{defn:Inv}
For a permutation $\alpha \in \asp$, the set of \emph{inversions} of $\alpha$ is 
\hfill\leanlink[AspPerm]{inv_set}
$$\Inv \alpha = \{ (u,v) \in \ZZ^2:\ u<v \mbox{ and } \alpha(u) > \alpha(v) \}.$$
\end{defn}

The set of inversions leads to two other preorders on $\asp$.
The terms below are chosen to mirror terminology from Coxeter groups; see e.g. \cite{bjornerBrenti}, especially \S 1.4, Exercise 1.13, \S 3.1, Proposition 3.1.3, and Corollary 3.1.4.

\begin{defn}
\label{defn:weakOrders}
The \emph{left weak preorder} on $\asp$ is the relation $\leq_L$, where $\alpha \leq_L \beta$ if and only if $\Inv \alpha \subseteq \Inv \beta$. 
The \emph{right weak preorder} $\leq_R$ is defined by $\alpha \leq_R \beta$ if and only if $\Inv \alpha^{-1} \subseteq \Inv \beta^{-1}$, i.e. $\alpha^{-1} \leq_L \beta^{-1}$.
\end{defn}

These are not partial orders on all of $\asp$: for instance, $\iota_m \leq_L \iota_n$ and $\iota_n \leq_L \iota_m$ for all $m,n \in \ZZ$. Proposition \ref{prop:reconstruction} shows that they become partial orders after restricting to permutations with fixed shift.

Closely related to the weak orders is the notion of a reduced product.

\begin{defn}
\label{defn:reducedProduct}
A product $\alpha \beta$ of two permutations is called \emph{reduced} if $\Inv (\alpha) \cap \Inv(\beta^{-1}) = \emptyset$. We write ``$\alpha \starr \beta$'' as a shorthand for ``$\alpha \beta$ is a reduced product.'' The symbol $\starr$ will sometimes be used in a larger statement, in which case it denotes the same thing as $\star$ and simultaneously asserts that the product is reduced; so $\alpha \starr \beta = \gamma$ abbreviates ``$\alpha \star \beta = \gamma$ and $\Inv (\alpha) \cap \Inv (\beta^{-1}) = \emptyset$.''
\hfill\leanlink{AspPerm.ReducedProduct}
\end{defn}

In fact, we prove in Lemma \ref{lem:reducedStar} that $\alpha \beta$ is a reduced product if and only if $\alpha \star \beta = \alpha \beta$, which is why we have chosen the notation $\starr$.

The two weak orders and the notion of reduced products are closely related.

\begin{lemma}
\label{lem:reducedWeakEquivs}
For $\alpha, \beta \in \asp$, the following are equivalent.
\hfill\leanbox{\doclink{AspPerm.reduced_iff_leR}\;\doclink{AspPerm.reduced_iff_leL}}
\begin{multicols}{3}
\begin{enumerate}
\item $\alpha \starr \beta$
\item $\alpha \leq_R \alpha \beta$
\item $\beta \leq_L \alpha \beta$
\end{enumerate}
\end{multicols}
\end{lemma}

\begin{proof}
A pair $(u,v) \in \ZZ^2$ belongs to $\Inv \alpha \cap \Inv \beta^{-1}$ if and only if $$(\alpha(v), \alpha(u)) \in \Inv \alpha^{-1} \backslash \Inv \beta^{-1} \alpha^{-1},$$ and similarly $(u,v) \in \Inv \alpha \cap \Inv \beta^{-1}$ if and only if $(\beta^{-1}(v), \beta^{-1}(u)) \in \Inv \beta \backslash \Inv \alpha \beta$. So the three sets $\Inv \alpha \cap \Inv \beta^{-1}$, $\Inv \alpha^{-1} \backslash \Inv \beta^{-1} \alpha^{-1}$, $\Inv \beta \backslash \Inv \alpha \beta$ are either all empty or all nonempty.
\end{proof}

\begin{rem}
Lemma \ref{lem:invStar} and Theorem \ref{thm:reduce} will imply the following alternative description of the weak orders: $\alpha \leq_L \beta$ (resp. $\alpha \leq_R \beta$) if and only if there exists some $\gamma \in \asp$ such that $\gamma \star \alpha = \beta$ (resp. $\alpha \star \gamma = \beta$). This gives a way to remember which is ``left'' and ``right:'' it depends on which side the extra permutation $\gamma$ is placed.
\end{rem}

\begin{warning}
\label{warning:weakBruhat}
The Bruhat order on $\asp$ has some counterintuitive features when used to compare permutations of different shift. For example, Lemma \ref{lem:bruhatInverse} is false without the assumption $\ca = \cb$; a counterexample is $\iota_0 \leq \iota_1$ but $\iota_0 \geq \iota_{-1}$.
In contrast to the situation in $S_d$ or other Coxeter subgroups of $\asp$, and indeed seemingly against basic decency of word choice, $\alpha \leq_L \beta$ does \emph{not} imply that $\alpha \leq \beta$ in $\asp$, so $\leq_L$ is not actually a ``weaker'' partial order than $\leq$. For example, $\iota_m \leq_L \iota_n$ for all $m,n \in \ZZ$. However, this implication is valid if we assume $\chi_\alpha = \chi_\beta$ (Corollary \ref{cor:weakStrong}).
\end{warning}

\subsection{Reconstruction from the inversion set} \label{ssec:recon}
This subsection is not logically needed elsewhere in the present paper. It is included to give another conceptual view on $\asp$: the permutations in $\asp$ are in bijection with the pairs consisting of a shift $\chi$ and a subset $I \subseteq \ZZ^2$ satisfying some simple assumptions. Characterizations of this kind are known in many similar contexts; see for example \cite[\S 2.3]{hohlwegLabbe}. The utility is that this provides a convenient route to efficiently describing the set of all permutations below a fixed $\alpha \in \asp$ in the weak preorders.

The first observation is that the shift and inversion set uniquely determine the permutation.

\begin{prop} \label{prop:reconstruction}
Any $\alpha \in \asp$ is determined by $\Inv \alpha$ and $\ca$ as follows. For all $n \in \ZZ$,
\hfill\leanlink{AspPerm.reconstruction}
\[ \alpha(n) = n - \ca + \# \left\{ v \in \ZZ:\ (n,v) \in \Inv \alpha \right\} - \# \left\{ u \in \ZZ:\ (u,n) \in \Inv \alpha \right\}. \]
\end{prop}

\begin{proof}
Observe that $\sa(\alpha(n),n) = \# \{ v \geq n:\ \alpha(v) < \alpha(n) \} = \# \{ v \in \ZZ:\ (n,v) \in \Inv \alpha \}$, and $\sai(n,\alpha(n)) = \# \{ u < n:\ \alpha(u) \geq \alpha(n) \} = \# \{ u \in \ZZ:\ (u,n) \in \Inv \alpha \}$. Equation \eqref{eq:saDuality} now rearranges to the reconstruction identity.
\end{proof}

So mapping a permutation $\alpha$ to the pair $(\ca, \Inv \alpha)$ is an injective map, and one may ask for a characterization of its image. Such a characterization is happily quite simple.

\begin{defn} \label{defn:aspSet}
An \emph{ASP set} is a subset $I \subseteq \ZZ^2$ with the following
properties.
\hfill\leanlink[InvSet]{AspSet}
\begin{enumerate}
\item If $(u,v) \in I$, then $u<v$.
\item If $(u,v),(v,w) \in I$, then $(u,w) \in I$ ($I$ is \emph{closed}).
\item If $u<v<w$ and $(u,v),(v,w) \notin I$, then $(u,w) \notin I$ ($I$ is \emph{coclosed}).
\item For every $n \in \ZZ$, the two sets
\[
\Out_I(n) = \{v \in \ZZ:\ (n,v) \in I\},
\qquad
\In_I(n) = \{u \in \ZZ:\ (u,n) \in I\}
\]
are finite ($I$ is \emph{locally finite}).
\end{enumerate}

\end{defn}

\begin{thm}
\label{thm:aspSetReconstruction}
Let $I \subseteq \ZZ^2$ and let $\chi \in \ZZ$. 
There exists a unique $\alpha \in \asp$ such that $\Inv \alpha = I$ and $\ca = \chi$ if and only if $I$ is an ASP set. 
Consequently, the map
\(
\alpha \longmapsto (\Inv\alpha,\ca)
\)
is a bijection from $\asp$ to the set of pairs consisting of an ASP set and an
integer.
\hfill\leanbox{\doclink[InvSet]{AspSet.invSets_of_AspPerms}\;\doclink[InvSet]{AspSet.AspPerm_equiv_AspSet}}
\end{thm}

\begin{proof}
It is straightforward to see that, if such an $\alpha$ exists, then $I$ is an ASP set, and the reconstruction formula of Proposition \ref{prop:reconstruction} shows that $\alpha$ is unique. Therefore it suffices to prove existence when $I$ is an ASP set.
Fix an ASP set $I$ and an integer $\chi$.  Define a function $\alpha : \ZZ \to \ZZ$ in the only possible way consistent with the reconstruction formula:
\begin{equation}\label{eq:aspFromSet}
  \alpha(n) = n - \chi + \#\Out_I(n) - \#\In_I(n).
\end{equation}

Define also a relation $\prec_I$ on $\ZZ$ as follows:
\[
m \prec_I n
\quad \Longleftrightarrow \quad
\big(m<n \text{ and } (m,n) \notin I\big)
\quad \text{or} \quad
\big(n<m \text{ and } (n,m) \in I\big).
\]

We claim that $\prec_I$ is a strict linear order on $\ZZ$. For any $a,b \in \ZZ$, the definition immediately shows that exactly one of $a \prec_I b$, $a = b$, or $b \prec_I a$ holds, so it suffices to show that $\prec_I$ is transitive. This follows from the ``closed'' and ``coclosed'' conditions on $I$, by doing casework on the possible relative orderings of three integers $a,b,c$ satisfying $a \prec_I b$ and $b \prec_I c$.

Having defined this alternative linear order on $\ZZ$, we may use it to define $I$-intervals:
\[ [m,n)_I = \{\ell \in \ZZ \colon m \preceq_I \ell \prec_I n\}, \]
where $\preceq_I$ denotes the reflexive closure of $\prec_I$.
When $m<n$, both $[m,n)_I$ and $[n,m)_I$ are finite: the first is contained in $[m,n) \cup \In_I(m) \cup \Out_I(n)$, and the second is contained in $[m,n) \cup \In_I(n) \cup \Out_I(m)$. The key observation is the following identity of cardinalities. Let $m,n$ be integers with $m < n$. Noting that all sums below have finite support, we have

\begin{align*}
\# [m,n)_I - \# [n,m)_I &= \sum_{\ell \in \ZZ} \left[ \delta(\ell \prec_I n) - \delta(\ell \prec_I m) \right]\\
&= \sum_{\ell \in \ZZ} \Big[ \delta( \ell \in \Out_I(n)) + \delta( \ell < n) - \delta( \ell \in \In_I(n)) \\
&\hspace{1cm} - \delta(\ell \in \Out_I(m)) - \delta(\ell < m) + \delta(\ell \in \In_I(m))
\Big]\\
&= \# [m,n) + \# \Out_I(n) - \# \In_I(n)  - \# \Out_I(m) + \# \In_I(m)\\
&= \alpha(n) - \alpha(m).
\end{align*}

Applying this identity with the smaller of $m,n$ first, and swapping $m,n$ if necessary, gives the following formula, valid for all distinct $m,n \in \ZZ$.
\[ \alpha(n) - \alpha(m) = \begin{cases} \# [m,n)_I & \text{if } m \prec_I n,\\
- \# [n,m)_I & \text{if } n \prec_I m.
\end{cases} \]

From this, it immediately follows that for all $m,n \in \ZZ$, $\alpha(m) < \alpha(n)$ if and only if $m \prec_I n$.
In particular $\alpha$ is injective. To see that $\alpha$ is surjective, observe first that local finiteness lets us build arbitrarily large and small values of $\alpha$: from any $x$ one can choose $y>x$ with $y \notin \Out_I(x)$, giving $x \prec_I y$ and hence $\alpha(x)<\alpha(y)$, and similarly choose $y<x$ with $y \notin \In_I(x)$, giving $y \prec_I x$ and hence $\alpha(y)<\alpha(x)$. Thus for any $k \in \ZZ$, there exist $m,n \in \ZZ$ with $\alpha(m) \le k < \alpha(n)$. Now $\alpha$ maps $[m,n)_I$ injectively to the interval $[\alpha(m),\alpha(n))$, which has the same cardinality, so $k$ must be in its image.

For $u<v$, the order comparison above gives $(u,v) \in I$ if and only if $\alpha(u)>\alpha(v)$, so $\Inv(\alpha)=I$. The definition of ASP set implies that $\alpha$ is almost-sign-preserving. Comparing Equation \eqref{eq:aspFromSet} with the reconstruction identity shows that $\alpha$ has shift $\chi$.
\end{proof}

\section{Slipface functions}
\label{sec:slipface}

In this section, we define a broader class of functions $\ZZ^2 \to \NN$, possessing several of the features ascribed to the slipface functions $\sa$. We call this broader class \emph{slipfaces} as well, and will characterize those slipface functions equal to $\sa$ for some $\alpha \in \asp$ as \emph{submodular slipfaces.} There are two reasons to step back to this more general setting. First, many of our proofs are naturally organized by first reasoning about general slipfaces and then adding the submodularity condition. Second, in our intended application to metric graphs, non-submodular slipfaces naturally occur.

\begin{defn}
\label{defn:slipface}
Let $\chi$ be an integer. A \emph{slipface function of shift $\chi$} is a function $s: \ZZ^2 \to \NN$ such that for all $a,b \in \ZZ$,
\begin{enumerate}[label=(S\arabic*)]
\item $0 \leq s(a+1,b) - s(a,b) \leq 1$ and $0 \leq s(a,b) - s(a,b+1) \leq 1$;
\item $s(a,b) \geq \max \{ 0, \chi + a - b \}$;
\item Given $b \in \ZZ$, for all but finitely many $a' \in \ZZ$, $s(a',b) = \max\{0, \chi + a' - b\}$. Given $a \in \ZZ$, for all but finitely many $b' \in \ZZ$, $s(a,b') = \max\{0, \chi + a - b'\}$.
\end{enumerate}
A function $s$ is called a \emph{slipface} if it is a slipface of some shift; the shift is uniquely determined by $s$ and denoted $\chi_s$. Denote by $\slip_\chi$ the set of slipfaces of shift $\chi$, and $\slip$ the set of all slipfaces. 
\hfill\leanlink[SlipFace]{SlipFace}
\end{defn}

\begin{defn}
Partially order $\slip$ as follows: $s \leq t$ means that $s(a,b) \leq t(a,b)$ for all $a,b \in \ZZ$. We call this the \emph{Bruhat order} on $\slip$.
\hfill\leanlink{SlipFace.instPartialOrder}
\end{defn}

\begin{eg}
\label{eg:saSlipface}
For $\alpha \in \asp$, the discussion in Section \ref{sec:aspPrelim} shows that $\sa \in \slip_{\ca}$. For axiom (S3), the almost-sign-preserving condition forces $\sa(a,b)$ to be $0$ far enough in the directions $b \gg 0$ and $a \ll 0$, and Equation \eqref{eq:saDuality} gives equality in (S2) in the opposite directions.
\hfill\leanlink{AspPerm.s}
\end{eg}

Slipfaces possess a duality, analogous to Equation \eqref{eq:saDuality}.

\begin{defn}
\label{defn:sdual}
If $s \in \slip_\chi$, then the \emph{dual slipface} is the function $s^\vee \in \slip_{-\chi}$ characterized by
\hfill\leanlink{SlipFace.dual}
$$s(a,b) - s^\vee(b,a) = \chi + a - b.$$
\end{defn}

Substituting into Definition \ref{defn:slipface} shows $s^\vee$ is indeed a slipface, with shift $-\chi$.

\begin{eg}
\label{eg:sai}
Equation \eqref{eq:saDuality} demonstrates that, for all $\alpha \in \asp$, 
\hfill\leanlink{AspPerm.s_dual}
$$\sa^\vee = \sai.$$
\end{eg}

One convenient aspect of duality is that it provides a streamlined criterion for checking that a given function is indeed a slipface, and determining its shift, in many situations. For a function $f: \ZZ^2 \to \ZZ$, we consider the following two criteria on $f$.
\begin{enumerate}[label=(D\arabic*)]
\item For all $a,b \in \ZZ$, $f(a+1,b) \geq f(a,b)$ and $f(a,b+1) \leq f(a,b)$.
\item For fixed $a$ and any $b \gg 0$, $f(a,b) = 0$; similarly, for fixed $b$ and any $a \ll 0$, $f(a,b) = 0$.
\end{enumerate}
\begin{lemma}
\label{lem:dualCrit}
Suppose that $s,t$ are two functions $\ZZ^2 \to \ZZ$ and $\chi$ is an integer, satisfying the equation
$$ s(a,b) - t(b,a) = \chi + a - b$$
for all $a,b \in \ZZ$. Then $s \in \slip_\chi$ if and only if both $s$ and $t$ satisfy conditions (D1) and (D2). If so, then also $t \in \slip_{-\chi}$ and $t = s^\vee$.
\hfill\leanbox{\doclink{SlipFace.sf_of_D_props}\;\doclink{SlipFace.D_props_of_sf}}
\end{lemma}

\begin{proof}
Criterion (D1) is equivalent to the lower bounds in (S1) when applied to $s$, and the upper bounds in (S1) when applied to $t$. Criterion (S2) is equivalent to both $s$ and $t$ being nonnegative, and equality holds in (S2) whenever either $s$ or $t$ vanishes, so (S3) follows from (D2) applied to both $s$ and $t$. Conversely, (D1) follows from (S1) and (D2) follows from (S3).
\end{proof}

\subsection{The operations $\star, \resL, \resR$ on $\slip$}
We are concerned with the following operations on $\slip$.

\begin{defn}
\label{defn:sfAlgebra}
Let $s, t \in \slip$. Define functions $s \star t, s \resL t$ and $s \resR t$ as follows.
\begin{flalign*}
&& s \star t(a,b) &=\min_{\ell \in \ZZ}\ \Big[ s(a,\ell) + t(\ell,b) \Big] & \leanlink{SlipFace.star_eq_min}\\
&& s \resL t (a,b) &= \max_{\ell \in \ZZ}\ \Big[ s(a,\ell) - t^\vee(b,\ell) \Big] & \leanlink{SlipFace.lres_eq_max}\\
&& s \resR t (a,b) &= \max_{\ell \in \ZZ}\ \Big[ t(\ell,b) - s^\vee(\ell,a) \Big] & \leanlink{SlipFace.rres_eq_max}
\end{flalign*}
\end{defn}

To see that the maxima defining $\resL$ and $\resR$ exist, note that as $k \to \pm \infty$, the fact that $s,t$ are slipfaces means that the expressions $s(a,k) - t^\vee(b,k)$ and $s^\vee(k,b) - t(k,a)$ converge, with one of the two limits converging to $0$. This implies that the maxima are well-defined, and nonnegative.

The following compatibilities with $\leq$ are immediate from this definition.

\begin{lemma}
\label{lem:compatLeq}
The product $s \star t$ is nondecreasing in both $s$ and $t$ (in Bruhat order). That is, if $s_1 \leq s_2$ and $t_1 \leq t_2$, then $s_1 \star t_1 \leq s_2 \star t_2$. On the other hand, $s \resL t^\vee$ and $t^\vee \resR s$ are both nondecreasing in $s$ and \emph{nonincreasing} in $t$.
\end{lemma}
\hfill\leanbox{\doclink{SlipFace.star_mono}\;\doclink{SlipFace.lres_mono}\;\doclink{SlipFace.rres_mono}}

\begin{prop}
\label{prop:sfAlgebraDefined}
Let $s, t \in \slip$. Then $s \star t, s \resL t$ and $s \resR t$ are all slipfaces of shift $\chi_s + \chi_t$. These operations satisfy the identities $(s \star t)^\vee = t^\vee \star s^\vee$ and $(s \resL t)^\vee = t^\vee \resR s^\vee$.
\hfill\leanbox{\doclink{SlipFace.chi_star}\;\doclink{SlipFace.chi_lres}\;\doclink{SlipFace.chi_rres}\;\doclink{SlipFace.star_dual}\;\doclink{SlipFace.lres_dual}\;\doclink{SlipFace.rres_dual}}
\end{prop}

\begin{proof}
Consider first the operation $\star$. Note first that $s \star t(a,b)$ is well-defined and nonnegative for all $a,b$ since $s,t$ are nonnegative. Writing $s(a,\ell) = s^\vee(\ell,a) + \chi_s + a -\ell$ and $t(\ell,b) = t^\vee(b,\ell) + \chi_t + \ell - b$ and unwinding definitions shows that
$$s \star t (a,b) = t^\vee \star s^\vee(b,a) + \chi_s + \chi_t + a - b$$
for all $a, b \in \ZZ$. By Lemma \ref{lem:dualCrit}, it suffices to verify that $s \star t$ and $t^\vee \star s^\vee$ satisfy (D1) and (D2). We need only check this for $s \star t$ since the same argument will also apply to $t^\vee \star s^\vee$. Criterion (D1) follows from $s(a+1,\ell) \geq s(a,\ell)$ and $t(\ell,b+1) \leq t(\ell,b)$. To verify (D2), fix $a \in \ZZ$. For $\ell \gg 0$, $s(a,\ell) = 0$; fixing such a $\ell$ and taking $b \gg 0 $, $t(\ell,b) = 0$ and thus $s \star t(a,b) = 0$. A similar argument shows that for any $b \in \ZZ$, $s \star t(a,b) = 0$ for all $a \ll 0$. This verifies (D2), completing the proof that $s \star t \in \slip_{\chi_s + \chi_t}$, with dual $t^\vee \star s^\vee$.

Now consider the operations $\resL$ and $\resR$.
We claim that for all $s,t \in \slip$, both $s \resL t$ and $s \resR t$ satisfy conditions (D1) and (D2). Criterion (D1) follows from the fact that $s,t$ and their duals satisfy (D1); (D2) requires a bit of casework. To verify (D2) for $s \resL t$, let $L(a,b)$ denote the set of integers $\ell$ such that $s(a,\ell) > t^\vee(b,\ell)$. Observe that $L(a,b+1) \subseteq L(a,b)$ and $L(a-1,b) \subseteq L(a,b)$; we must show that if one of $a,b$ is fixed and the other is decreased or increased (respectively) the set $L(a,b)$ eventually shrinks to $\emptyset$. This follows from two observations. First, $L(a,b)$ is finite provided that $\chi_s+\chi_t + a - b \leq 0$, by considering the limit of $s(a,\ell) - t^\vee(b,\ell)$ as $\ell$ tends to $\pm \infty$. Second, for any particular $a,b,\ell \in \ZZ$, $\ell$ eventually drops out of $L(a,b)$ as $a$ is decreased or $b$ is increased. More precisely, for $a,\ell$ fixed and $b$ sufficiently large, $s(a,\ell) - t^\vee(b,\ell) = s(a,\ell) - \chi_t - b + \ell \leq 0$ and thus $\ell \not\in L(a,b)$; similarly for $\ell,b$ fixed and $a$ sufficiently small, $s(a,\ell) - t^\vee(b,\ell) = 0 - t^\vee(b,\ell) \leq 0$ and thus $\ell \not\in L(a,b)$. Therefore $s \resL t$ satisfies (D2). A similar argument shows that $s \resR t$ satisfies (D2): we may define $L(a,b) = \{\ell \in \ZZ: t(\ell,b) < s^\vee(\ell,a)\}$; then $L(a,b)$ is finite provided that $\chi_s + \chi_t + a - b \leq 0$, and any particular $\ell \in L(a,b)$ eventually drops out as $a$ is decreased or $b$ is increased.

We now claim that $s \resL t$ and $t^\vee \resR s^\vee$ are dual slipfaces, with $\chi_{s \resL t} = \chi_s + \chi_t$. Writing $s(a,\ell) = s^\vee(\ell,a) + \chi_s + a - \ell$ and $t^\vee(b,\ell) = t(\ell,b) -\chi_t + b - \ell$ and unwinding definitions shows that
$$s \resL t (a,b) = t^\vee \resR s^\vee(b,a) + \chi_s + \chi_t  + a - b$$
for all $a,b \in \ZZ$. Since $s \resL t$ and $t^\vee \resR s^\vee$ both satisfy (D1) and (D2), Lemma \ref{lem:dualCrit} implies that they are dual slipfaces with $\chi_{s \resL t} = \chi_s + \chi_t$. Replacing $s,t$ by $t^\vee, s^\vee$ in this argument shows that $s \resR t$ is also a slipface of shift $\chi_s + \chi_t$ and dual $t^\vee \resL s^\vee$.
\end{proof}

\begin{lemma}
\label{lem:sfOpChar}
Let $s,t,u \in \slip$. The following are equivalent.
\hfill\leanbox{\doclink{SlipFace.ge_star_iff_ge_rres}\;\doclink{SlipFace.ge_star_iff_ge_lres}}
\begin{multicols}{3}
\begin{enumerate}
\item $s \star t \geq u$
\item $s \geq u \resL t^\vee$
\item $t \geq s^\vee \resR u$
\end{enumerate}
\end{multicols}
In other words, the operators $\resL, \resR$ are the following minima in Bruhat order.
$$
u \resL t^\vee = \min \{ s \in \slip: s \star t \geq u\}, \mbox{ and } 
s^\vee \resR u = \min \{t \in \slip: s \star t \geq u\}.
$$
\end{lemma}

\begin{proof}
Each is equivalent to $s(a,\ell) + t(\ell,b) \geq u(a,b)$ holding for all $a,b,\ell \in \ZZ$.
\end{proof}

\begin{rem}
\label{rem:residuated}
Lemma~\ref{lem:sfOpChar} expresses a \emph{residuation property} for $\slip$, equipped with $\le$ and $\star$.
This means that, if we regard $\slip$ as a category with a unique morphism $s \to t$ whenever $s \le t$, then the functor $\bullet \star t$ has adjoint $\bullet \resL t^\vee$, and likewise $s \star \bullet$ has adjoint $s^\vee \resR \bullet$. This is the reason for the choice of terminology \emph{residual} for $\resL$ and $\resR$.
\end{rem}

\begin{lemma}
\label{lem:sfAlgebra}
The operations $\star, \resL, \resR$ on $\slip$ satisfy the following identities.
\begin{enumerate}
\item $s \star (t \star u) = (s \star t ) \star u$ ($\star$ is associative). \hfill \leanlink{SlipFace.star_assoc}
\item $(s \resL t) \resL u = s \resL (t \star u)$. \hfill \leanlink{SlipFace.lres_assoc}
\item $s \resR (t \resR u) = (s \star t) \resR u$. \hfill \leanlink{SlipFace.rres_assoc}
\end{enumerate}
\end{lemma}

\begin{proof}
Part (1) follows by observing that for all $a,b \in \ZZ$,
$$s \star (t \star u)(a,b) = \min_{\ell_1,\ell_2 \in \ZZ} \Big[ s(a,\ell_1) + t(\ell_1,\ell_2) + u(\ell_2,b) \Big] = (s \star t) \star u(a,b).$$
Parts (2) and (3) follow from combining part (1) with the characterization of $\resL, \resR$ in Lemma \ref{lem:sfOpChar}. For part (2), note that $(s \resL t) \resL u$ is the minimum $v \in \slip$ such that $v \star u^\vee \geq s \resL t$, which is the minimum such that $(v \star u^\vee) \star t^\vee \geq s$. By associativity of $\star$ and the equation $u^\vee \star t^\vee = (t \star u)^\vee$, this is equivalent to $v \star (t \star u)^\vee \geq s$, so $v = s \resL (t \star u)$. An analogous argument proves part (3).
\end{proof}

\subsection{Example computations}
\label{ssec:slipSpecial}

We now illustrate the definitions of $\star$ and $\resL$ in two particularly simple cases: when one permutation is increasing, or a disjoint product, possibly infinite, of adjacent transpositions. In these and other computations, it is useful to shrink the set on which the extremum defining $\star, \resL,$ or $\resR$ is computed. The following lemma provides one way to do so (by no means the only way). For convenience this lemma is stated only for $\star$ and $\resL$, but it can be applied to $\resR$ as well using the identity $s \resR t = (t^\vee \resL s^\vee)^\vee$.

\begin{lemma}
\label{lem:setL}
Fix $s, t \in \slip$ and $a,b \in \ZZ$. Define a set \hfill\leanlink{SlipFace.bend_set}
$$L = \{ \ell \in \ZZ: t(\ell-1,b) = t(\ell,b) < t(\ell+1,b) \}.$$
Then the minimum defining $s \star t(a,b)$, and the maximum defining $s \resL t(a,b)$, each occur for some $\ell \in L$. That is, $\displaystyle s \star t (a,b) = \min_{\ell \in L} \Big[ s(a,\ell) + t(\ell,b) \Big]$, and $s \resL t(a,b)  =\ds  \max_{\ell \in L} \Big[ s(a,\ell) - t^\vee(b,\ell) \Big]$.
\hfill\leanbox{\doclink{SlipFace.bend_set_witness}\;\doclink{SlipFace.bend_set_witness_lres}}

The set $L$ is finite.
In the case $t = \sbe$ for some $\beta \in \asp$, it is also equal to
\hfill\leanlink{SlipFace.bend_set_finite}
\leanEquation{}{%
L = \{ \ell \in \ZZ: \beta^{-1}(\ell-1) < b \leq \beta^{-1}(\ell) \}.}{%
\leanlink{AspPerm.bend_set_sf}}
\end{lemma}

\begin{proof}
The finiteness of $L$ follows from criterion (S3), and the alternative formula in the case $t = \sbe$ follows from Equation \eqref{eq:a+1}.

First consider $s \star t$. 
Let $M$ be the set of integers $\ell$ where the minimum occurs, i.e. 
$$M = \left\{ \ell \in \ZZ: s \star t (a,b) = s(a,\ell) + t(\ell,b) \right\}.$$
It suffices to show that $L \cap M$ is nonempty.
Criterion (S1) implies the following two facts.
\begin{enumerate}
\item If $\ell \in M$ and $t(\ell-1,b) < t(\ell,b)$, then $\ell-1 \in M$.
\item If $\ell \in M$ and $t(\ell+1,b) = t(\ell,b)$, then $\ell+1 \in M$.
\end{enumerate}
Let $M_1 \subset M$ denote the subset on which $t(\ell,b)$ is minimized among all elements of $M$. This is bounded above since $t(\ell,b)$ tends to infinity as $\ell$ grows. Let $\ell$ be the maximum element of $M_1$. Then properties (1) and (2) above imply that $\ell \in L$. So $L$ and $M$ have nonempty intersection.

Now consider $\resL$, and redefine $M = \{\ell \in \ZZ: s \resL t(a,b) = s(a,\ell) - t^\vee(b,\ell) \}$. Using the fact that $t^\vee(b,\ell+1) < t^\vee(b,\ell)$ if and only if $t(\ell+1,b) = t(\ell,b)$, it follows that properties (1) and (2) above hold exactly as stated for this set $M$, and the same argument shows that $M$ intersects $L$.
\end{proof}

\begin{lemma}
\label{lem:starShift}
Let $s \in \slip$ and $n \in \ZZ$. For all $a,b \in \ZZ$,
$$
s \star \si{n}(a,b) = s \resL \si{n}(a,b) = s(a,b-n).
$$
in particular, $\si{0}$ is the identity element for $\star$ on $\slip$.
\end{lemma}

\begin{proof}
Let $t = \si{n}$ in Lemma \ref{lem:setL}. Then $t(a,b) = \max \{0, n + a - b\}$ for all $a,b \in \ZZ$, so $L = \{b-n\}$.
Therefore $s \star \si{n}(a,b) = s(a,b-n) + \si{n}(b-n,b)$ and $s \resL \si{n}(a,b) = s(a,b-n) - \si{n}^\vee(b,b-n)$. Using $\si{n}^\vee = \si{-n}$, both $\si{n}(b-n,b)$ and $\si{n}^\vee(b,b-n)$ are zero, and the displayed equation follows. So $s \star \si{0} = s$; it also follows that $\si{0} \star t = (t^\vee \star \si{0})^\vee = (t^\vee)^\vee = t$ for all $t \in \slip$, so $\si{0}$ is both a left-identity and a right-identity for $\star$.
\end{proof}

\begin{cor}
\label{cor:stGeqS}
If $s,t \in \slip$ and $\chi_t \geq 0$, then $s \star t \geq s$ and $s \resL t \leq s.$
\end{cor}

\begin{proof}
Lemma \ref{lem:compatLeq} implies that $s \star t \geq s \star \si{0} = s$ and $s \resL t^\vee \leq s \resL \si{0}^\vee = s$.
\end{proof}

\begin{defn}
\label{defn:simpleInv}
Let $S \subseteq \ZZ$ be a set of integers that contains no two consecutive integers. Let $\sigma_S: \ZZ \to \ZZ$ be the permutation exchanging $n$ and $n+1$ for all $n \in S$, and fixing all other integers.
In other words, $\sigma_S$ is the (possibly infinite) product of the simple transpositions $\{ \sigma_n: n \in S \}$.
\hfill\leanlink{Transpositions.sigma}
\end{defn}

These include the simple transpositions $\sigma_n = \sigma_{\{n\}}$ used to generate symmetric groups, as well as the involutions $\tsig_n = \sigma_{n + k \ZZ}$ used to generate the affine symmetric group $\ts_k$.

\newcommand{\sss}{s_{\sigma_S}}
\begin{lemma}
\label{lem:starTrans}
Let $S \subset \ZZ$ be a set containing no two consecutive integers, and let $t = \sss$ be the slipface associated to the permutation defined above. Let $s$ be any other slipface. For all $a,b \in \ZZ$,
\hfill\leanbox{\doclink{Transpositions.sf_star_sigma}\;\doclink{Transpositions.sf_lres_sigma}}
\begin{eqnarray*}
s \star \sss(a,b) &=& s(a,b) + \delta\Big[ b-1 \in S \mbox{ and } s(a,b-1) > s(a,b) = s(a,b+1)\Big],
\mbox{ and }\\
s \resL \sss(a,b) &=& s(a,b) - \delta\Big[ b-1 \in S \mbox{ and } s(a,b-1) = s(a,b) > s(a,b+1)\Big].
\end{eqnarray*}
In the special case where $s = \sa$ for some $\alpha \in \asp$,
\begin{eqnarray*}
\sa \star \sss(a,b) &=& \sa(a,b) + \delta\Big[ b-1 \in S \mbox{ and } \alpha(b-1) < a \leq \alpha(b)\Big],
\mbox{ and}\\
\sa \resL \sss(a,b) &=& \sa(a,b) - \delta\Big[ b-1 \in S \mbox{ and } \alpha(b) < a \leq \alpha(b-1)\Big].
\end{eqnarray*}
\end{lemma}
\begin{proof}
The second two equations, in the special case $s = \sa$, follow from the general case and Equation \eqref{eq:b+1}, so we focus on the general case.
Fix $a,b \in \ZZ$. Note that $\sigma_S^{-1} = \sigma_S$ and thus $\sss^\vee = \sss$. The set $L$ from Lemma \ref{lem:setL} is
$$L = \{ \ell \in \ZZ: \sigma_S(\ell-1) < b \leq \sigma_S(\ell) \} 
= \begin{cases}
\{b\} & \mbox{ if } b-1 \not\in S,\\
\{b-1, b+1\} & \mbox{ if } b-1 \in S.
\end{cases}$$
If $b-1 \not\in S$, then $\sss(b,b) = 0$, so $s \star \sss(a,b) = s \resL \sss(a,b) = s(a,b)$, as desired. On the other hand, if $b-1 \in S$, then $\sss(b-1,b) = \sss^\vee(b,b+1) = 0$ and $\sss(b+1,b) = \sss^\vee(b,b-1) = 1$, so $s \star \sss(a,b) = \min\{ s(a,b-1), s(a,b+1)+1\}$, and $s \resL \sss(a,b) = \max\{ s(a,b-1)-1, s(a,b+1)\}$. The result now follows from criterion (S1).
\end{proof}

\begin{cor}
\label{cor:sasss}
If $\alpha \in \asp$ has $\alpha(n) > \alpha(n+1)$ for all $n \in S$, then $\sa \star \sss = \sa$. In particular, $\sss \star \sss = \sss$. On the other hand, if $\alpha(n) < \alpha(n+1)$ for all $n \in S$, then $\sa \resL \sss = \sa$.
\end{cor}

A more general formula for $\alpha \star \sigma_S$ is proved in Theorem \ref{thm:alphaStarSigma}.

\section{\texorpdfstring{The image of $\asp$ in $\slip$: submodular slipfaces}{The image of asp in slip: submodular slipfaces}}
\label{sec:submodular}

The map $\alpha \mapsto \sa$ embeds $\asp \hookrightarrow \slip$. We demonstrate in this section that the image is closed under the operations $\star, \resL, \resR$, and thereby obtain corresponding operations on $\asp$. A key tool is a characterization of the image of this inclusion as the set of \emph{submodular} slipfaces.

The following shorthand will prove useful. For $s \in \slip$, define $\Delta s: \ZZ^2 \to \{-1,0,1\}$ by
$$\Delta s(a,b) = s(a+1,b) - s(a,b) - s(a+1,b+1) + s(a,b+1).$$
The fact that $\Delta s(a,b) \in \{-1,0,1\}$ follows from Criterion (S1). The function $s$ can be reconstructed from $\Delta s$. Observe that for $a,b$ fixed, there are only finitely many $(a',b') \in \ZZ$ such that $a' \leq a, b' \geq b$ and $s(a',b') > 0$. From this and a telescoping sum, it follows that
\begin{equation}
\label{eq:sumDelta}
s(a,b) = \sum_{ \substack{a' < a\\b' \geq b} } \Delta s(a',b').
\end{equation}
Here and throughout this section, an equation involving a sum of infinitely many integers is understood to mean implicitly that only finitely many terms are nonzero; equivalently, the sum converges absolutely.
The definition of $s^\vee$ and some cancellation shows that 
\begin{equation}
\label{eq:DeltasDual}
\Delta s^\vee(a,b) = \Delta s(b,a),
\end{equation}
and then Equation \eqref{eq:sumDelta} implies that
\begin{equation}
\label{eq:sumDeltaDual}
s^\vee(b,a) = \sum_{ \substack{a' \geq a\\b' < b} } \Delta s(a',b').
\end{equation}
Therefore any slipface $s$ can be reconstructed from the knowledge of the points $(a,b)$ where $\Delta s(a,b) =1$ and the points where $\Delta s(a,b) = -1$. These two sets are fairly constrained, however. Since $\displaystyle \lim_{b \to \infty} \left( s(a+1,b) - s(a,b) \right) = 1$, and $\displaystyle \lim_{a \to \infty} \left( s^\vee(b+1,a) - s^\vee(b,a) \right) = 1$, we deduce

\begin{lemma}
\label{lem:DeltaBalance}
For fixed $a \in \ZZ$, $\sum_{n \in \ZZ} \Delta s(a,n) = 1$. For fixed $b \in \ZZ$, $\sum_{n \in \ZZ} \Delta s(n,b) = 1$.
\end{lemma}

\begin{defn}
\label{defn:submodular}
A slipface $s$ is \emph{submodular} if $\Delta s(a,b) \geq 0$ for all $a,b \in \ZZ$.
\hfill\leanlink{SlipFace.submodular}
\end{defn}

\begin{prop}
\label{prop:imageASP}
A slipface $s$ is submodular if and only if there exists $\alpha \in \asp$ such that $s = \sa$. Therefore $\alpha \mapsto \sa$ gives a bijection from $\asp$ to the set of submodular slipfaces.
\hfill\leanlink{Submodular.submodular_iff_asp}
\end{prop}

\begin{proof}
Equation \eqref{eq:Deltasa} shows that $\sa$ is submodular for all $\alpha \in \asp$. Conversely, suppose that $s \in \slip$ is submodular. Let
$ \Gamma = \{ (a,b): \Delta s(a,b) = 1 \}.$
Lemma \ref{lem:DeltaBalance} implies that $\Gamma$ contains a unique element for each value of $a$, and a unique element for each value of $b$. Therefore there exists a (unique) \emph{permutation} $\alpha: \ZZ \to \ZZ$ such that $(a,b) \in \Gamma$ if and only if $\alpha(b) = a$. Equations \eqref{eq:sumDelta} and \eqref{eq:sumDeltaDual} may be written $s(a,b) = \sa(a,b)$ and $s^\vee(b,a) = \sai(b,a)$; since these are finite, $\alpha \in \asp$.
\end{proof}

\begin{rem} \label{rem:asm}
If we view $s$ and $\Delta s$ as $|\ZZ| \times |\ZZ|$ matrices, then Equation \eqref{eq:sumDelta} shows that each matrix determines the other. When $s$ is submodular, we may view $\Delta s$ a an infinite permutation matrix. On the other hand, for any slipface $s$, $\Delta s$ is an infinite \emph{alternating sign matrix}. This means that the nonzero entries in every row or column alternate between $1$ and $-1$, and each row or column sums to $1$.
Both these properties follow from the observation that the sum of a finite contiguous sequence of entries in a row or column telescopes to $-1,0$, or $1$ by Criterion (S1), and it tends to $1$ as the sequence grows on both sides by Criterion (S3).
Alternating sign matrices have rich combinatorics; see for example the book \cite{BressoudASM} or discussion in \cite{fischerSaikia}.
\end{rem}

\subsection{The Demazure product $\star$ on $\asp$}

This subsection proves the following theorem and develops a few combinatorial properties of $\star$ on $\asp$.

\begin{thm}
\label{thm:starExists1}
The set of submodular slipfaces is closed under $\star$. Therefore there is a well-defined associative product $\star$ on $\asp$, uniquely characterized by $\sab = \sa \star \sbe$, which we call the \emph{Demazure product} on $\asp$.
It is associative, and satisfies $(\alpha \star \beta)^{-1} = \beta^{-1} \star \alpha^{-1}$ and $\chi_{\alpha \star \beta} = \ca + \cb$.
\end{thm}
\hfill\leanbox{\doclink{Submodular.submodular_of_star}\;\doclink[Submodular]{AspPerm.star_spec}\;\doclink[Submodular]{AspPerm.star_assoc}\;\doclink[Submodular]{AspPerm.inverse_star}\;\doclink[Submodular]{AspPerm.chi_star}}

\begin{defn}
Let $\alpha, \beta \in \asp$ and $a,b \in \ZZ$. Define 
$$\mab(a,b) = \max \{\ell \in \ZZ: \sa \star \sbe (a,b) = \sa(a,\ell) + \sbe(\ell,b) \}.$$
\end{defn}

We will use the following elementary lemma several times. We omit the straightforward proof.

\begin{lemma}
\label{lem:fg}
For a function $f: \ZZ \to \ZZ$ that is bounded below, denote $\min f = \min \{ f(\ell): \ell \in \ZZ\}$ and let $M_f \in \ZZ \cup \{+\infty\}$ be $M_f = \sup \{ \ell \in \ZZ:\ f(\ell) = \min f \}$. Suppose that $A$ is any integer, and $g: \ZZ \to \ZZ$ is defined by 
$g(\ell) = f(\ell) + \delta(\ell \leq A).$
Then 
$\min g = \min f + \delta( M_f \leq A), \mbox{ and } M_g \geq M_f.$
\hfill\leanlink{Submodular.sediment}
\end{lemma}

\begin{lemma}
\label{lem:Kstara+1}
For all $a,b \in \ZZ$, 
$\sa \star \sbe(a+1,b) = \sa \star \sbe(a,b) + \delta\Big[ \mab(a,b) \leq \alpha^{-1}(a) \Big],$ and
$\mab(a+1,b) \geq \mab(a,b)$.
\hfill\leanlink{Submodular.AspValley_step_a}
\end{lemma}

\begin{proof}
Let $f(\ell) = \sa (a,\ell) + \sbe(\ell,b)$ and $g(\ell) = \sa(a+1,\ell) + \sbe(\ell,b)$. By Equation \eqref{eq:a+1},
$$g(\ell) = f(\ell) + \delta( \ell \leq \alpha^{-1}(a)).$$
Observe $\min f = \sa \star \sbe(a,b)$ and $\min g = \sa \star \sbe(a+1,b)$, and apply Lemma \ref{lem:fg}.
\end{proof}

\begin{lemma}
\label{lem:Kstarb+1}
For all $a,b \in \ZZ$, 
$\sa \star \sbe(a,b+1) = \sa \star \sbe(a,b) - \delta\Big[ \mab(a,b) > \beta(b) \Big],$ and
$\mab(a,b+1) \geq \mab(a,b)$.
\hfill\leanlink{Submodular.AspValley_step_b}
\end{lemma}

\begin{proof}
Let $f(\ell) = \sa(a,\ell) + \sbe(\ell,b) - 1$ and $g(\ell) = \sa(a,\ell) + \sbe(\ell,b+1)$. 
Equation \eqref{eq:b+1} implies that
$$g(\ell) = f(\ell) + 1 - \delta(\ell > \beta(b)) = f(\ell) + \delta(\ell \leq \beta(b)).$$
Observe that $\min f = \sa \star \sbe(a,b) -1$ and $\min g = \sa \star \sbe(a,b+1)$, and apply Lemma \ref{lem:fg}.
\end{proof}

\begin{proof}[Proof of Theorem \ref{thm:starExists1}]
In light of Proposition \ref{prop:imageASP}, we must show that for any two $\alpha, \beta \in \asp$, the slipface $\sa \star \sbe$ is submodular. Using property (S1), it suffices to show that, if $\sa \star \sbe(a+1,b) = \sa \star \sbe(a+1,b+1)$, then also $\sa \star \sbe(a,b) = \sa \star \sbe(a,b+1)$. By Lemma \ref{lem:Kstarb+1}, we must show that if $\mab(a+1,b) \leq \beta(b)$, then also $\mab(a,b) \leq \beta(b)$. This follows from $\mab(a,b) \leq \mab(a+1,b)$ (Lemma \ref{lem:Kstara+1}). This establishes that submodular slipfaces are closed under $\star$, so $\star$ is well-defined on $\asp$. The remaining claims follow from the identities $\sa^\vee = s_{\alpha^{-1}}$ and $\chi_{\sa} = \ca$ (Examples \ref{eg:saSlipface} and \ref{eg:sai}) and the corresponding identities on slipfaces (Proposition \ref{prop:sfAlgebraDefined} and Lemma \ref{lem:sfAlgebra}). 
\end{proof}

The discussion above also provides information about the inversions of $\alpha \star \beta$.

\begin{lemma}
\label{lem:invStar}
For any $\alpha, \beta \in \asp$, $\beta \leq_L \alpha \star \beta$ and $\alpha \leq_R \alpha \star \beta$.
\hfill\leanbox{\doclink{Submodular.lel_of_dprod}\;\doclink{Submodular.ler_of_dprod}}
\end{lemma}

\begin{proof}
It suffices to prove only the first statement; the second is equivalent to $\alpha^{-1} \leq_L \beta^{-1} \star \alpha^{-1}$. We must prove that $\Inv(\alpha \star \beta) \supseteq \Inv(\beta)$. We argue by contrapositive. Suppose that $u,v \in \ZZ$ satisfy $u < v$ and $\alpha \star \beta(u) < \alpha \star \beta(v)$. We will prove that $\beta(u) < \beta(v)$.

Choose any integer $a$ such that $\alpha \star \beta(u) < a \leq \alpha \star \beta(v)$. By Equation \eqref{eq:b+1}, this is equivalent to $\sab(a,u+1) = \sab(a,u)-1$ and $\sab(a,v+1) = \sab(a,v)$. By Lemma \ref{lem:Kstarb+1}, this means that $\mab(a,u) > \beta(u)$ and $\mab(a,v) \leq \beta(v)$. Lemma \ref{lem:Kstarb+1} also implies that $\mab(a,u) \leq \mab(a,v)$. Chaining these inequalities implies that $\beta(u) < \beta(v)$. 
\end{proof}

\subsection{The operations $\resL, \resR$ on $\asp$}
This subsection parallels the previous. We will prove

\begin{thm}
\label{thm:resLExists}
The set of submodular slipfaces is closed under $\resL$ and $\resR$.
Therefore there are well-defined binary operations $\resL, \resR$ on $\asp$ characterized by $\sa \resL \sbe = s_{\alpha \resL \beta}$ and $\sa \resR \sbe = s_{\alpha \resR \beta}$.
These operations satisfy the following identities:
$(\alpha \resL \beta) \resL \gamma = \alpha \resL (\beta \star \gamma)$,
$\alpha \resR (\beta \resR \gamma) = (\alpha \star \beta) \resR \gamma$,
$(\alpha \resL \beta)^{-1} = \beta^{-1} \resR \alpha^{-1}$,
and
$\chi_{\alpha \resL \beta} = \chi_{\alpha \resR \beta} = \ca + \cb$, as well as Equations \eqref{eq:resLMin} and \eqref{eq:resRMin}.
\end{thm}
\hfill\leanbox{\doclink{Submodular.submodular_of_lres}\;\doclink{Submodular.submodular_of_rres}\;\doclink[Submodular]{AspPerm.lres_spec}\;\doclink[Submodular]{AspPerm.rres_spec}}

\hfill\leanbox{\doclink[Submodular]{AspPerm.lres_assoc}\;\doclink[Submodular]{AspPerm.rres_assoc}\;\doclink[Submodular]{AspPerm.inverse_lres}\;\doclink[Submodular]{AspPerm.chi_lres}\;\doclink[Submodular]{AspPerm.chi_rres}\;\doclink[Submodular]{AspPerm.lres_eq_min}\;\doclink[Submodular]{AspPerm.rres_eq_min}}

\begin{defn}
Let $\alpha, \beta \in \asp$ and $a,b \in \ZZ$. Define
$$ \mtab(a,b) = \sup \{ \ell \in \ZZ:\ \sa \resL \sbe(a,b) = \sa(a,\ell) - \sbei(b,\ell) \}.$$
If $\sa \resL \sbe (a,b) = 0$, then this is understood to be $+\infty$.
\end{defn}

\begin{lemma}
\label{lem:Mta+1}
For all $a,b \in \ZZ$, $\sa \resL \sbe(a+1,b) = \sa \resL \sbe(a,b) + \delta( \mtab(a+1,b) \leq \alpha^{-1}(a) )$, and $\mtab(a+1,b) \leq \mtab(a,b)$.
\end{lemma}

\begin{proof}
Define $f(\ell) = - \sa(a+1,\ell) + \sbei(b,\ell)$ and $g(\ell) = -\sa(a,\ell) + \sbei(b,\ell)$. By Equation \eqref{eq:a+1},
$$g(\ell) = f(\ell) + \delta( \ell \leq \alpha^{-1}(a) ).$$
Observe $\min g = - \sa \resL \sbe(a,b)$ and $\min f = - \sa \resL \sbe(a+1,b)$, and apply Lemma \ref{lem:fg}.
\end{proof}

\begin{lemma}
\label{lem:Mtb+1}
For all $a,b \in \ZZ$, $\sa \resL \sbe(a,b+1) = \sa \resL \sbe(a,b) - \delta(\mtab(a,b) \leq \beta(b))$, and $\mtab(a,b+1) \geq \mtab(a,b)$.
\end{lemma}

\begin{proof}
Define $f(\ell) = -\sa(a,\ell) + \sbei(b,\ell)$ and $g(\ell) = -\sa(a,\ell) + \sbei(b+1,\ell)$. By Equation \eqref{eq:a+1},
$$g(\ell) = f(\ell) + \delta( \ell \leq \beta(b) ).$$
Observe $\min f = - \sa \resL \sbe(a,b)$ and $\min g = - \sa \resL \sbe(a,b+1)$, and apply Lemma \ref{lem:fg}.
\end{proof}

\begin{proof}[Proof of Theorem \ref{thm:resLExists}]
We first check closure. By Proposition \ref{prop:imageASP}, it suffices to show that $\sa \resL \sbe$ and $\sa \resR \sbe$ are submodular for all $\alpha, \beta \in \asp$. Since $\sa \resR \sbe = (s_{\beta^{-1}} \resL s_{\alpha^{-1}})^\vee$, it suffices to check that $\sa \resL \sbe$ is submodular for all $\alpha, \beta \in \asp$. Fix permutations $\alpha, \beta \in \asp$.

This establishes that $\resL$ and $\resR$ are well-defined on $\asp$. The identities follow from the identities $\sa^\vee = s_{\alpha^{-1}}$ and $\chi_{\sa} = \ca$ (Examples \ref{eg:saSlipface} and \ref{eg:sai}) and the corresponding identities on slipfaces, established in Proposition \ref{prop:sfAlgebraDefined} and Lemma \ref{lem:sfAlgebra}. Equations \eqref{eq:resLMin} and \eqref{eq:resRMin} follow from Lemma \ref{lem:sfOpChar}.
\end{proof}

The analog of Lemma \ref{lem:invStar} for the operations $\resL, \resR$ is the following.

\begin{lemma}
\label{lem:invRes}
For all $\alpha, \beta \in \asp$, $(\alpha \resL \beta) \starr \beta^{-1}$, and 
$\alpha \resL \beta \leq_R \alpha$.
\hfill\leanbox{\doclink{Submodular.reducedProduct_of_lres}\;\doclink{Submodular.ler_of_lres}}
\end{lemma}

\begin{proof}
First, suppose that $(u,v) \in \Inv(\alpha \resL \beta)$. We will show that $(u,v) \not\in \Inv(\beta)$. There exists an integer $a$ such that $\alpha \resL \beta(v) < a \leq \alpha \resL \beta(u)$. By Equation \eqref{eq:b+1}, $\sabl(a,v+1) = \sabl(a,v)-1$ and $\sabl(a,u) = \sabl(a,u+1)$, whence Lemma \ref{lem:Mtb+1} implies that $\mtab(a,v) \leq \beta(v)$, $\mtab(a,u) > \beta(u)$. Since $u < v$, Lemma \ref{lem:Mtb+1} also implies that $\mtab(a,v) \geq \mtab(a,u)$. Chaining these inequalities implies $\beta(u) < \beta(v)$, as desired.

\newcommand{\sbail}{s_{\beta^{-1} \resL \alpha^{-1}}}
Now suppose that $(u,v) \in \Inv((\alpha \resL \beta)^{-1})$. We will show that $(u,v) \in \Inv(\alpha^{-1})$. Similarly to above, there exists $b$ such that $\sabl(v+1,b) = \sabl(v,b)$ and $\sabl(u+1,b) = \sabl(u,b)+1$. Lemma \ref{lem:Mta+1} implies that $\mtab(v+1,b) > \alpha^{-1}(v)$ and $\mtab(u+1,b) \leq \alpha^{-1}(u)$. Since $u<v$, Lemma \ref{lem:Mta+1} also implies that $\mtab(v+1,b) \leq \mtab(u+1,b)$. Chaining these inequalities gives $\alpha^{-1}(u) > \alpha^{-1}(v)$, as desired.
\end{proof}

\begin{cor}
\label{cor:reducedResR}
For all $\alpha, \beta \in \asp$, $\alpha^{-1} \starr (\alpha \resR \beta)$ and $\alpha \resR \beta \leq_L \beta$.
\hfill\leanbox{\doclink{Submodular.reducedProduct_of_rres}\;\doclink{Submodular.lel_of_rres}}
\end{cor}
\begin{proof}
These statements are equivalent to $(\beta^{-1} \resL \alpha^{-1}) \starr \alpha$ and $\beta^{-1} \resL \alpha^{-1} \leq_R \beta^{-1}$.
\end{proof}

\section{\texorpdfstring{When $\star, \resL, \resR$ are ordinary products}{When Demazure product and residual are ordinary products}}
\label{sec:reducedWeak}

This section examines the relation between $\star, \resL, \resR$ on $\asp$ and the ordinary product.

\begin{lemma}
\label{lem:reducedStar}
For all $\alpha, \beta \in \asp$, $\alpha \star \beta \geq \alpha \beta$, and $\alpha \star \beta = \alpha \beta$ if and only if $\alpha \starr \beta$.
\par\noindent\hfill\leanbox{\doclink{ReducedProducts.mul_le_star}\;\doclink{ReducedProducts.star_eq_mul_iff_reducedProduct}}
\end{lemma}

\begin{proof}
Observe that for all $a,b,\ell \in \ZZ$,
$$
\sa(a,\ell) + \sbe(\ell,b) = \# \{n \geq \ell: \alpha(n) < a \} + \# \{n < \ell: \beta^{-1}(n) \geq b \}.
$$
The two sets counted on the right side are disjoint. Partitioning their union a different way, the right side is equal to
\begin{eqnarray*}
\# \{n \in \ZZ: \alpha(n) < a \mbox{ and } \beta^{-1}(n) \geq b \}
+ \# \{n \geq \ell: \alpha(n) < a \mbox{ and } \beta^{-1}(n) < b \}\\
+ \# \{n < \ell: \alpha(n) \geq a \mbox{ and } \beta^{-1}(n) \geq b \}.
\end{eqnarray*}
The first of these three terms is equal to $s_{\alpha \beta}(a,b)$; it follows from this that 
$$\sa(a,\ell) + \sbe(\ell,b) \geq s_{\alpha \beta}(a,b) \mbox{ for all } a,b,\ell \in \ZZ.$$ In other words, $\alpha \star \beta \geq \alpha \beta$. This argument also shows how to describe the equality case: $\alpha \star \beta = \alpha \beta$ if and only if for all $a,b \in \ZZ$, there exists an integer $\ell \in \ZZ$ such that the second and third terms in the sum above vanish. This amounts to saying that every element of $\{n \in \ZZ: \alpha(n) < a \mbox{ and } \beta^{-1}(n) < b\}$ is less than every element of $\{n \in \ZZ: \alpha(n) \geq a \mbox{ and } \beta^{-1}(n) \geq b\}$.

Put another way, $\alpha \star \beta > \alpha \beta$ if and only if there exist two integers $a,b$ and two integers $m < n$ such that $\alpha(n) < a \leq \alpha(m)$ and $\beta^{-1}(n) < b \leq \beta^{-1}(m)$. For fixed $m,n$, there exist $a,b$ satisfying these chains if and only if $\alpha(n) < \alpha(m)$ and $\beta^{-1}(n) < \beta^{-1}(m)$. Since we require $m < n$, this simply means $(m,n) \in \Inv (\alpha) \cap \Inv(\beta^{-1})$. So $\alpha \star \beta > \alpha \beta$ if and only if $\Inv (\alpha) \cap \Inv(\beta^{-1})$ is nonempty.
\end{proof}

\begin{lemma}
\label{lem:reducedRes}
For all $\alpha, \beta \in \asp$, $\alpha \resL \beta \leq \alpha \beta$ and $\alpha \resR \beta \leq \alpha \beta$. Furthermore, $\alpha \resL \beta = \alpha \beta$ if and only if $\beta^{-1} \leq_L \alpha$ and $\alpha \resR \beta = \alpha \beta$ if and only if $\alpha^{-1} \leq_R \beta$.
\hfill\leanbox{\doclink{ReducedProducts.lres_le_mul}\;\doclink{ReducedProducts.lres_eq_mul_iff}\;\doclink{ReducedProducts.rres_le_mul}\;\doclink{ReducedProducts.rres_eq_mul_iff}}
\end{lemma}

\begin{proof}
We will prove that $\alpha \resL \beta \leq \alpha \beta$, with equality if and only if $\Inv(\beta^{-1}) \subseteq \Inv(\alpha)$. The corresponding statements about $\resR$ will then follow from the identity $\alpha \resR \beta = (\beta^{-1} \resL \alpha^{-1})^{-1}$. The strategy is analogous to that of Lemma \ref{lem:reducedStar}, although the counting argument is slightly more subtle.

Observe that for all $a,b,\ell \in \ZZ$,
\begin{eqnarray*}
\sa(a,\ell) - \sbei(b,\ell) &=& \# \{n \geq \ell: \alpha(n) < a \} - \# \{n \geq \ell: \beta^{-1}(n) < b \}\\
&=& \# \{ n \geq \ell: \alpha(n) < a,\ \beta^{-1}(n) \geq b \} - \# \{n \geq \ell: \alpha(n) \geq a,\ \beta^{-1}(n) < b \}
\end{eqnarray*}
In the second line, the intersection of the two sets on the right side has been removed from both. From here, we may rewrite the first term in the last line as
\begin{eqnarray*}
\# \{n \geq \ell: \alpha(n) < a,\ \beta^{-1}(n) \geq b \} &=& \# \{n \in \ZZ: \alpha(n) < a,\ \beta^{-1}(n) \geq b\}
\\&&
- \# \{n < \ell: \alpha(n) < a,\ \beta^{-1}(n) \geq b \}\\
&=& s_{\alpha \beta}(a,b) - \# \{n < \ell: \alpha(n) < a,\ \beta^{-1}(n) \geq b \}.
\end{eqnarray*}
Therefore 
\begin{eqnarray*}
\sa(a,\ell) - \sbei(b,\ell)\ =\ s_{\alpha \beta}(a,b) &-& \# \{n < \ell: \alpha(n) < a,\ \beta^{-1}(n) \geq b \}\\
 &-& \# \{n \geq \ell: \alpha(n) \geq a,\ \beta^{-1}(n) < b \}.
\end{eqnarray*}
Therefore $\sa(a,\ell) - \sbei(b,\ell) \leq s_{\alpha \beta}(a,b)$ for all $a,b,\ell \in \ZZ$, with equality case characterized by two sets being empty, much as in the proof of Lemma \ref{lem:reducedStar}. Therefore $\alpha \resL \beta \leq \alpha \beta$. Studying the equality case as in Lemma \ref{lem:reducedStar}, $\alpha \resL \beta < \alpha \beta$ if and only if there exist integers $m < n$ and integers $a,b$ such that $\alpha(m) < a \leq \alpha(n)$ and $\beta^{-1}(m) \geq b > \beta^{-1}(n)$; this condition is equivalent to $\Inv(\beta^{-1}) \backslash \Inv(\alpha)$ being nonempty. So $\alpha \resL \beta = \alpha \beta$ if and only if $\Inv(\beta^{-1}) \subseteq \Inv(\alpha)$.
\end{proof}

Among other things, this clarifies the relationship between the weak and strong Bruhat orders.

\begin{cor}
\label{cor:weakStrong}
If either $\alpha \leq_L \beta$ or $\alpha \leq_R \beta$, and $\chi_\alpha \leq \chi_\beta$, then $\alpha \leq \beta$.
\hfill\leanbox{\doclink{ReducedProducts.le_of_le_weak_L_of_chi_le}\;\doclink{ReducedProducts.le_of_le_weak_R_of_chi_le}}
\end{cor}

\begin{proof}
If $\alpha \leq_L \beta$ and $\ca \leq \cb$, then Lemma \ref{lem:reducedWeakEquivs} implies $\beta = (\beta \alpha^{-1}) \starr \alpha$. Corollary \ref{cor:stGeqS} implies that $\beta \geq \alpha$, since $\chi_{\beta \alpha^{-1}} = \chi_\beta - \ca \geq 0$. Similarly, if $\alpha \leq_R \beta$, then $\beta = \alpha \starr (\alpha^{-1} \beta) \geq \alpha$.
\end{proof}

\section{Greediness, stinginess, and the reduction theorem}
\label{sec:mainTheorems}

The tools of the previous section furnish proofs of the greedy characterization of $\star$, the ``stingy'' characterizations of $\resL$ and $\resR$, and the reduction theorem.

\begin{proof}[Proof of Theorem \ref{thm:starGreedy}]
Fix $\alpha, \beta \in \asp$.
Lemmas \ref{lem:compatLeq} and \ref{lem:reducedStar} imply that if $\alpha_1 \leq \alpha, \beta_1 \leq \beta$, then
\leanEquation{eq:astarbBound}{%
\alpha \star \beta \geq \alpha_1 \star \beta_1 \geq \alpha_1 \beta_1.}{%
\leanlink{Reduction.mul_le_star_of_le}}
It remains to exhibit some equality cases.
First, let $\alpha_1 = (\alpha \star \beta) \beta^{-1}$ and $\beta_1 = \beta$.
Since $\beta \leq_L \alpha \star \beta$ (Lemma \ref{lem:invStar}), Lemma \ref{lem:reducedRes} implies that $\alpha_1 = (\alpha \star \beta) \resL \beta^{-1}$.
 By Lemma \ref{lem:sfOpChar}, $\alpha_1$ is the minimum element of $\{ \gamma \in \asp: \gamma \star \beta \geq \alpha \star \beta \}$. Of course $\alpha$ belongs to this set, so $\alpha_1 \leq \alpha$. The shift of $\alpha_1$ is $\chi_{\alpha_1} = \chi_\alpha$, by Theorem \ref{thm:starExists1} and Equation \eqref{eq:chiHom}. Lemma \ref{lem:reducedRes} implies $\alpha_1 \starr \beta = \alpha_1 \beta = \alpha \star \beta$. So equality is obtained in Equation \eqref{eq:astarbBound} in a case where $\beta_1 = \beta$ and $\chi_{\alpha_1} = \chi_\alpha$; this proves Equations \eqref{eq:starGreedyAlpha} and \eqref{eq:starGreedy} from the theorem. The remaining Equation \eqref{eq:starGreedyBeta} follows by applying Equation \eqref{eq:starGreedyAlpha} to $\beta^{-1} \star \alpha^{-1}$, plus the identity $(\beta_1^{-1} \star \alpha^{-1})^{-1} = \alpha \star \beta_1$ and Lemma \ref{lem:bruhatInverse}.
\end{proof}

The analogy of the ``greedy theorem'' for $\star$ is the following ``stingy theorem'' for $\resL$.

\begin{thm}
\label{thm:resLStingy}
For all $\alpha, \beta \in \asp$, the following minima exist in Bruhat order, and equal $\alpha \resL \beta^{-1}$.
\begin{flalign*}
\eqnum{eq:resLGreedy}
&& \alpha \resL \beta^{-1} &= \min \left\{ \alpha_1 \beta_1^{-1}:\ \alpha_1 \geq \alpha,\ \beta_1 \leq \beta \right\}
& \leanlink{Reduction.lres_stingy}
\\
\eqnum{eq:resLGreedyBeta}
&& &= \min \left\{ \alpha \beta_1^{-1}: \beta_1 \lechi \beta \right\}
& \leanlink{Reduction.lres_stingy_beta}
\\
\eqnum{eq:resLGreedyAlpha}
&& &= \min \left\{ \alpha_1 \beta^{-1}: \alpha_1 \geq_\chi \alpha \right\}.
& \leanlink{Reduction.lres_stingy_alpha}
\end{flalign*}
\end{thm}

\begin{proof}
Fix $\alpha, \beta \in \asp$. Lemmas \ref{lem:compatLeq} and \ref{lem:reducedRes} imply that if $\alpha_1 \geq \alpha$ and $\beta_1 \leq \beta$, then
\begin{equation}
\label{eq:aresLbBound}
\alpha \resL \beta^{-1} \leq \alpha_1 \resL \beta_1^{-1} \leq \alpha_1 \beta_1^{-1}.
\end{equation}
We must now find some equality cases. 
First, define $\alpha_1 = \alpha$ and $\beta_1 = (\beta \resR \alpha^{-1}) \alpha$. So $\alpha \resL \beta^{-1} = \alpha \beta_1^{-1}$ and $\chi_{\beta_1} = \chi_\beta$. We must show that $\beta_1 \leq \beta$. We may also write $\beta_1 = (\beta \resR \alpha^{-1}) \resR \alpha$, by Corollary \ref{cor:reducedResR} and Lemma \ref{lem:reducedRes}. So $\beta_1$ is the minimum permutation such that 
$ \beta_1 \star \alpha^{-1} \geq \beta \resL \alpha^{-1}$. But $\beta$ is another such permutation, since 
$\beta_1 \star \alpha^{-1} \geq \beta_1 \alpha^{-1} = (\alpha \resL \beta^{-1})^{-1} = \beta \resR \alpha^{-1}$.
Therefore $\beta \geq \beta_1$.

For the second equality case, define $\alpha_1 = (\alpha \resL \beta^{-1}) \beta$. Lemma \ref{lem:invRes} implies $(\alpha \resL \beta^{-1}) \starr \beta = \alpha_1$, and Lemma \ref{lem:sfOpChar} implies $(\alpha \resL \beta^{-1}) \star \beta \geq \alpha$. Therefore $\alpha_1 \geq_\chi \alpha$ and $\alpha_1 \beta = \alpha \resL \beta^{-1}$.
\end{proof}

The reduction theorem from the introduction follows by similar techniques.

\begin{proof}[Proof of Theorem \ref{thm:reduce}]
We will prove only the second paragraph of the theorem statement, since the first follows directly from it, along with the monotonicity property in Lemma \ref{lem:compatLeq}.

Let $\alpha, \beta, \gamma \in \asp$ satisfy $\alpha \star \beta \geq \gamma$, and define $\alpha_1 = \gamma \resL \beta^{-1}$ and $\beta_1 = \alpha_1^{-1} \resR \gamma$.
We apply Lemma \ref{lem:sfOpChar} (specialized to $\asp$) several times in a row. The assumed bound $\alpha \star \beta \geq \gamma$ implies $\alpha \geq \gamma \resL \beta^{-1} = \alpha_1$. The bound $\alpha_1 \geq \gamma \resL \beta^{-1}$ implies $\alpha_1 \star \beta \geq \gamma$, which implies $\beta \geq \alpha_1^{-1} \resR \gamma = \beta_1$. So $\alpha_1, \beta_1$ are bounded above by $\alpha, \beta$, respectively. They also have the same shifts, by Theorem \ref{thm:resLExists}.

Lemma \ref{lem:invRes} says $\alpha_1 \leq_R \gamma$, so  $\beta_1 = \alpha_1^{-1} \resR \gamma = \alpha_1^{-1} \gamma$ by Lemma \ref{lem:reducedRes}. The inequality $\alpha_1 \leq_R \gamma$ also implies via Lemma \ref{lem:reducedWeakEquivs} that $\alpha_1 \starr (\alpha_1^{-1} \gamma) = \gamma$. Combining these equations gives $\alpha_1 \starr \beta_1 = \gamma$.
\end{proof}

In some applications, it is convenient to use a form of Theorem \ref{thm:reduce} for products of more than two permutations; we give such a statement here for convenience. To state it, we must first clarify what is meant by a reduced product of three or more permutations; the definition below is based on the standard definition of a \emph{reduced word} in a Coxeter group, adapted both to products of non-generators and to permutations that may have infinitely many inversions.

\begin{lemma}
\label{lem:reducedProdSeveral}
Let $(\alpha_1, \cdots, \alpha_\ell)$ be an $\ell$-tuple of permutations in $\asp$. Denote by $\pi_n$ the product of the suffix $\alpha_{n+1} \cdots \alpha_\ell$ ($\pi_\ell$ is the identity). The following are equivalent.
\begin{enumerate}
\item For all $u,v \in \ZZ$ with $u \neq v$, the sign of $\pi_n(u) - \pi_n(v)$ changes at most once as $n$ decreases from $\ell$ to $0$.
\item The set $\Inv(\pi_0) = \Inv(\alpha_1 \cdots \alpha_\ell)$ is the \emph{disjoint} union of the sets 
$$I_n = \{ (\pi^{-1}_n(u), \pi^{-1}_n(v)):\ (u,v) \in \Inv(\alpha_n) \} \mbox{ for } 1 \leq n \leq \ell.$$
\item For each $1 \leq n \leq \ell$, $\alpha_n \starr \pi_n$.
\end{enumerate}
\end{lemma}

\begin{proof}
$(1) \Rightarrow (2)$: for all $u<v$, $(u,v) \in \Inv(\pi_0)$ if and only if $\pi_n(u) - \pi_n(v)$ changes sign an \emph{odd} number of times; assuming (1), this is equivalent to changing signs once, and in turn to belonging to exactly one (and necessarily only one) set $I_n$.

$(2) \Rightarrow (3)$: this follows since $\Inv(\alpha_n) \cap \Inv(\pi_n^{-1}) = \{ ( \pi_n(u), \pi_n(v)): (u,v) \in I_n\}$ and $\Inv(\pi_n^{-1})$ is contained in the union of $\{ (\pi_n(u), \pi_n(v)): (u,v) \in I_m \}$ for $n +1 \leq m \leq \ell$, which are disjoint.

$(3) \Rightarrow (1)$: if $(1)$ is false, then there exists some $u<v$ and $n$ such that $\pi_n(u) > \pi_n(v)$ but $\alpha_n \pi_n(u) < \alpha_n \pi_n(v)$; this implies that $( \pi_n(v), \pi_n(u)) \in \Inv(\alpha_n) \cap \Inv( \pi_n^{-1})$, so $(3)$ is false.
\end{proof}

\begin{defn}
An $\ell$-tuple $(\alpha_1, \cdots, \alpha_\ell)$ in $\asp$ is \emph{reduced} if the three equivalent conditions in Lemma \ref{lem:reducedProdSeveral} hold.
\end{defn}

\begin{rem}
Although we don't need it here, the following recursive criterion for reducedness is sometimes useful. For any $1 \leq i \leq \ell-1$, a tuple $(\alpha_1, \cdots, \alpha_\ell)$ is reduced if and only if both $(\alpha_1, \cdots, \alpha_i)$ and $(\alpha_{i+1}, \cdots, \alpha_\ell)$ are reduced tuples, and $(\alpha_1 \cdots \alpha_i) \starr (\alpha_{i+1} \cdots \alpha_\ell)$. This follows from criterion (1) in Lemma \ref{lem:reducedProdSeveral}.
\end{rem}

\begin{thm}
\label{thm:reduceSeveral}
Let $G \leq \asp$ be a subgroup that is closed under $\resL$. For any $\alpha_1, \cdots, \alpha_\ell, \gamma \in G$, if $\alpha_1 \star \cdots \star \alpha_\ell \geq_\chi \gamma$, then there exists a \emph{reduced} tuple $(\beta_1, \cdots, \beta_\ell)$ such that $\beta_i \in G$ and $\beta_i \lechi \alpha_i$ for all $i$, and
$\beta_1 \star \cdots \star \beta_\ell = \beta_1 \cdots \beta_\ell = \gamma.$
\end{thm}
\hfill\leanlink{Reduction.reduceList}

\begin{proof}
By induction on $\ell$. The base case $\ell = 1$ is tautological. If $\ell \geq 2$, then Theorem \ref{thm:reduce} implies that there exist $\beta_1, \pi \in G$ with $\beta_1 \starr \pi = \gamma$, $\beta_1 \lechi \alpha_1$, and $\alpha_2 \star \cdots \star \alpha_\ell \geq_\chi \pi$. By inductive hypothesis, there exist $\beta_2 \lechi \alpha_2, \cdots, \beta_\ell \lechi \alpha_\ell$ such that $(\beta_2, \cdots, \beta_\ell)$ reduced and $\beta_2 \star \cdots \star \beta_\ell = \beta_2 \cdots \beta_\ell = \pi$. Then $\beta_1 \star \cdots \star \beta_\ell = \gamma$, and $(\beta_1, \cdots, \beta_\ell)$ is reduced by criterion (3) of Lemma \ref{lem:reducedProdSeveral}.
\end{proof}

\section{Bounded-difference permutations and the essential set}
\label{sec:ess}

Having proved our main theorems, this section addresses a somewhat different question that will prove useful in applications: when comparing two permutations in Bruhat order, or more generally two slipfaces, does it suffice to check the inequality $s(a,b) \leq t(a,b)$ for some small subset of $\ZZ \times \ZZ$? There is a standard tool for this in $S_d$, called the \emph{essential set}, introduced in \cite{fultonSchubert} for applications to degeneracy loci, which we adapt herein to $\slip$ and $\asp$. Unfortunately, a subtle finiteness issue restricts the use of this tool to certain slipfaces, which we call \emph{Clifford slipfaces.} In the submodular case, these correspond to bounded-difference permutations, which we define below.

\subsection{The essential set of a slipface}
\label{ssec:ess}

The key definition is the following.

\begin{defn}
\label{defn:essSF}
For any $s \in \slip$, the \emph{essential set} of $s$ is
$$\ess(s) = \{ (a,b) \in \ZZ^2: 
s(a-1,b) < s(a,b) = s(a+1,b) \mbox{ and } s(a,b+1) < s(a,b) = s(a,b-1) \}.$$
\end{defn}

As in the theory of degeneracy loci, the purpose of $\ess(s)$ is to identify a smaller, often finite, set on which to verify inequalities $s(a,b) \leq t(a,b)$. The basic observation is

\begin{lemma}
\label{lem:essSetMoves}
Suppose that $s,t$ are slipfaces such that $s \not\leq t$, and let $E = \{ (a,b): s(a,b) > t(a,b) \}$ be the set witnessing this fact. If $(a,b) \in E$ and $(a,b) \not\in \ess(s)$, then there exists $(a',b') \in \{ (a-1,b), (a+1,b), (a,b-1), (a,b+1) \}$
such that $(a',b') \in E$ and
$s(a',b') + s^\vee(b',a') > s(a,b) + s^\vee(b,a).$
\hfill\leanlink{SlipFace.ess_step}
\end{lemma}

\begin{proof}
Define $f_s(a,b) = s(a,b) + s^\vee(b,a) + \chi_s$; equivalently $f_s(a,b) = 2 s(a,b) - a + b$. Then $s(a,b) > t(a,b)$ if and only if $f_s(a,b) > f_t(a,b)$. When one of $a,b$ is increased or decreased by one, one of $s(a,b), s^\vee(b,a)$ increases or decreases by $1$, while the other stays the same. So for any $(a',b')$ in the stated set, $f_s(a',b') = f_s(a,b) \pm 1$, and in the case $f_s(a',b') = f_s(a,b) +1$ we have $s(a',b') \geq s(a,b)$ and $s^\vee(b',a') \geq s^\vee(b,a)$. The assumption $(a,b) \not\in \ess(s)$ means precisely that $f_s(a',b') = f_s(a,b) + 1$ for some $(a',b')$ in the stated set. For such $(a',b')$, $f_t(a',b') = f_t(a,b) \pm 1$ implies that $f_s(a',b') > f_t(a',b')$, so $s(a',b') > t(a',b')$ and $(a',b') \in E$.
\end{proof}

Lemma \ref{lem:essSetMoves} \emph{almost} implies that $s \leq t$ if and only if $s(a,b) \leq t(a,b)$ for all $(a,b) \in \ess(s)$. The only catch is that the process of $(a,b)$ by $(a',b')$ might not terminate. A quick, if perhaps slightly unsatisfying, fix is simply to restrict the classes of slipfaces $s$ considered. Fortunately this restriction still encompasses all slipfaces that occur in our intended applications. The name \emph{Clifford} refers to Clifford's theorem from the theory of algebraic curves.

\begin{defn} \label{defn:cliffordSF}
A \emph{Clifford slipface} is a slipface $s$ such that the following set is bounded above.
$$\{ s(a,b) + s^\vee(b,a):\ s(a,b) > 0 \mbox{ and } s^\vee(b,a) > 0 \}$$
\end{defn}

\begin{prop}
\label{prop:essSF}
If $s$ is a Clifford slipface and $t$ is any slipface, then $s \leq t$ if and only if $\chi_s \le \chi_t$ and $s(a,b) \leq t(a,b)$ for all $(a,b) \in \ess(s)$.
\hfill\leanlink{SlipFace.ess_clifford}
\end{prop}

\begin{proof}
First suppose that $s \le t$. It suffices to check that $\chi_s \le \chi_t$. This follows from the fact that for $a \gg 0$, $s(a,0) = a + \chi_s$ and $t(a,0) = a + \chi_t$.

Now suppose that $s \not\leq t$. Define $E$ be the set of $(a,b)$ witnessing this fact, as in Lemma \ref{lem:essSetMoves}. 
If $E \cap \ess(s)$ is nonempty we are done, so assume that $E$ is disjoint from $\ess(s)$.
Lemma \ref{lem:essSetMoves} implies that $E$ has elements $(a,b)$ or arbitrarily large values of $s(a,b) + s^\vee(b,a)$. Since $s$ is Clifford, there exists $(a,b) \in E$ such that $s(a,b) = 0$ or $s^\vee(b,a) = 0$. Since $s(a,b) > t(a,b) \ge 0$, this means $s^\vee(b,a) = 0$. Therefore $s(a,b) = a - b + \chi_s$. Since $t(a,b) \ge a - b + \chi_t$, it follows that $\chi_s > \chi_t$.
\end{proof}

\begin{eg}
The slipface $s(a,b) = s_{\iota_\chi}(a,b) = \max \{ \chi+a-b,0 \}$ has $\ess(s) = \emptyset$. So Proposition \ref{prop:essSF} recovers again the fact that $s$ is the Bruhat-minimal slipface of shift $\chi_s$.
\end{eg}

\subsection{The essential set of a permutation}
Now consider the submodular case.

\begin{defn}
\label{defn:essPerm}
Let $\alpha \in \asp$. The \emph{essential set} of $\alpha$ is
$$\ess(\alpha) = \{ (a,b) \in \ZZ^2: \alpha^{-1}(a-1) \geq b > \alpha^{-1}(a),\ \alpha^{}(b-1) \geq a > \alpha^{}(b) \}.$$
\end{defn}

Equations \eqref{eq:b+1} and \eqref{eq:a+1} imply that 
\leanEquation{eq:essPermSa}{%
\ess(\sa) = \ess(\alpha).}{\doclink{AspPerm.ess_asp_eq_ess_sf}}
The expression in Definition \ref{defn:essPerm} matches the definition in \cite{fultonSchubert}, used in the study of degeneracy loci. For this to be useful, we must understand which $\alpha \in \asp$ have Clifford slipfaces $\sa$. If there is an integer $M$ such that $|\alpha(n) - n| \leq M$ for all $n \in \ZZ$, we will say that $\alpha$ \emph{has bounded difference}.

\begin{prop}
\label{prop:cliffordPerms}
Let $\alpha \in \asp$. The following are equivalent.
\hfill \leanbox{\doclink{AspPerm.bdiff_iff_width}\;\doclink{AspPerm.bdiff_iff_clifford}}
\begin{enumerate}
\item The permutation $\alpha$ has bounded difference.
\item There exists $N \in \ZZ$ such that $\sa(a,b) = \max \{\ca + a - b, 0\}$ whenever $|a-b| \geq N$.
\item The slipface $\sa$ is Clifford, i.e. $\{ \sa(a,b) + \sai(b,a):\  \sa(a,b) >0 \mbox{ and } \sai(b,a) > 0 \}$ is bounded.
\end{enumerate}
\end{prop}

\begin{lemma}
\label{lem:malpha}
Let $\alpha \in \asp$. The following are equal.
\hfill \leanbox{\doclink{AspPerm.M'_eq_M}\;\doclink{AspPerm.M''_eq_M}}
\begin{eqnarray}
\label{eq:supsai}
M_\alpha &=& \sup \{ \sai(b,a):\ \sa(a,b) > 0 \}\\
\label{eq:supdiff}
M_\alpha' &=& \sup \{ n - \alpha(n):\ n \in \ZZ \} - \ca\\
\label{eq:supba}
M_\alpha'' &=& \sup \{ b-a:\ \sa(a,b) > 0 \} -\ca + 1
\end{eqnarray}
\end{lemma}

\begin{proof}
We will prove that $M_\alpha \geq M'_\alpha \geq M''_\alpha \geq M_\alpha$.
First, fix any $n \in \ZZ$. Let 
$$m = \max \{ m \geq n:\ \alpha(m) \leq \alpha(n) \}.$$
Then $\sa(\alpha(m) +1, m) = 1$ and Equation \eqref{eq:saDuality} implies 
$$\sai(\alpha(m)+1,m) = m -  \alpha(m) - \ca.$$
Therefore
$M_\alpha \geq m - \alpha(m) - \ca \geq n - \alpha(n) - \ca.$
This shows that $M'_\alpha \leq M_\alpha$.

Next, fix any $a,b \in \ZZ$ such that $\sa(a,b) > 0$. Choose the maximum $n \geq b$ such that $\alpha(n) < a$. Then $b-a + 1 \leq n - \alpha(n) \leq M'_\alpha$. This implies $M''_\alpha \leq M'_\alpha$. On the other hand, choosing $a' = \alpha(n) + 1$ and $b'= n$, we have $\sa(a',b') = 1$ (by maximality of $n$) so $\sai(b',a') = 1-\ca + b'-a' \leq M''_\alpha$. This implies that $\sai(b,a) \leq \sai(b',a') \leq M''_\alpha$. Hence $M_\alpha \leq M''_\alpha$.
\end{proof}

\begin{proof}[Proof of Proposition \ref{prop:cliffordPerms}]
By Equation \eqref{eq:saDuality}, statement (2) is equivalent to saying that $\sa(a,b) = 0$ if $b-a \gg 0$ and $\sai(b,a) = 0$ if $a-b \gg 0$. By Lemma \ref{lem:malpha}, both statements (1) and (2) are equivalent to saying that both $M_\alpha$ and $M_{\alpha^{-1}}$ are finite. Now consider statement (3). 
The proposition is clear when $\alpha = \iota_{\ca}$, so assume $\alpha$ has at least one inversion. Define
$$C_\alpha = \sup \{\sa(a,b) + \sai(b,a):\ \sa(a,b) > 0 \mbox{ and } \sai(b,a) > 0 \}.$$
So statement (3) is equivalent to $C_\alpha < \infty$.
Observe that we have bounds 
$$1+ \max \{ M_\alpha, M_{\alpha^{-1}} \} \leq C_\alpha \leq M_{\alpha} + M_{\alpha^{-1}}.$$
So $C_\alpha$ is finite if and only if both $M_\alpha$ and $M_{\alpha^{-1}}$ are finite.
\end{proof}

\begin{cor}
\label{cor:essBD}
If $\alpha, \beta \in \asp$ have the same shift and $\alpha$ has bounded difference, then $\alpha \leq \beta$ if and only if $\sa(a,b) \leq \sbe(a,b)$ for all $(a,b) \in \ess(\alpha)$.
\hfill \leanlink[Submodular]{AspPerm.ess_bdiff}
\end{cor}

\begin{eg}
\label{eg:essPerm}
Observe that if $\alpha \in S_d$, then $\alpha$ has bounded difference. Both $\alpha$ and $\alpha^{-1}$ are increasing on $(-\infty,1] \cap \ZZ$ and $[d, \infty) \cap \ZZ$, so
$$\ess(\alpha) \subseteq \{ 2, \cdots, d \} \times \{ 2, \cdots, d \}.$$
This is not surprising: for a permutation moving only finitely many values, the Bruhat order depends on only a finite subset of $\ZZ^2$.
\end{eg}

\begin{eg}
\label{eg:bnPerm} 
Fix two nonnegative integers $m,n$. Consider the permutation $\gamma = \gamma^m_n$ depicted in Figure \ref{fig:bnPerm}. 
More precisely, $\gamma$ is the unique permutation mapping $(-\infty,-m-1] \cap \ZZ$ to $(-\infty, -n] \cap \ZZ$, $[-m,-1] \cap \ZZ$ to $[1,m] \cap \ZZ$, $[0,n-1] \cap \ZZ$ to $[-n+1, 0] \cap \ZZ$, and $[n,\infty) \cap \ZZ$ to $[m+1, \infty) \cap \ZZ$, each via the order-preserving bijection.

Examining the adjacent inversions of $\gamma$ shows that 
$$\ess(\gamma) = \{ (1,0) \}.$$
Furthermore, $s_\gamma(1,0) = n$ and $s_{\gamma^{-1}}(0,1) = m$, so $\chi_\gamma = n-m-1$. The set $\{ \gamma(n)-n:\ n \in \ZZ\}$ is finite, so it is certainly bounded, and $s_\gamma$ is Clifford. It follows that for any $\alpha \in \asp$ of shift $\chi_\gamma$,
$$ \sa(1,0) \geq n \mbox{ if and only if } \alpha \geq \gamma.$$
In a similar manner, any inequality of the form $\sa(a,b) \geq N$ may be expressed equivalently by a Bruhat inequality $\alpha \geq \gamma$ for some $\gamma$ of the same shift as $\alpha$. This construction is useful in the intended applications to curves and graphs; see Appendix \ref{app:curvesGraphs}.
\end{eg}

\begin{figure}
\begin{tikzpicture}[scale=0.2]
\fill[lightgray] (-0.5,0.5) rectangle (8,-8);
\fill[lightgray] (-0.5,0.5) rectangle (-8,8);

\draw[->] (-8,0) -- (8,0);
\draw[->] (0,-8) -- (0,8);

\node [ circle, inner sep=1pt] at (-3,-3) {};
\node [fill=black, circle, inner sep=1pt] at (-6,-7) {};
\node [fill=black, circle, inner sep=1pt] at (-5,-6) {};
\node [fill=black, circle, inner sep=1pt] at (-4,-5) {};
\node [fill=black, circle, inner sep=1pt] at (-3,1) {};
\node [fill=black, circle, inner sep=1pt] at (-2,2) {};
\node [fill=black, circle, inner sep=1pt] at (-1,3) {};
\node [fill=black, circle, inner sep=1pt] at (0,-4) {};
\node [fill=black, circle, inner sep=1pt] at (1,-3) {};
\node [fill=black, circle, inner sep=1pt] at (2,-2) {};
\node [fill=black, circle, inner sep=1pt] at (3,-1) {};
\node [fill=black, circle, inner sep=1pt] at (4,0) {};
\node [fill=black, circle, inner sep=1pt] at (5,4) {};
\node [fill=black, circle, inner sep=1pt] at (6,5) {};
\node [fill=black, circle, inner sep=1pt] at (7,6) {};

\draw[decoration={brace,mirror,raise=2pt},decorate]
  (8,-8) -- (8,0.5) node[midway,right] {$\ s_{\gamma}(1,0) = n$};
  
\draw[decoration={brace,mirror,raise=2pt},decorate] (-8,8) -- (-8,0.5) node[midway, left] {$s_{\gamma}(0,1) = m\ $};

\end{tikzpicture}

\caption{
The permutation $\gamma^m_n$ from Example \ref{eg:bnPerm}. The specific example shown is $\gamma^3_5$.
}
\label{fig:bnPerm}
\end{figure}
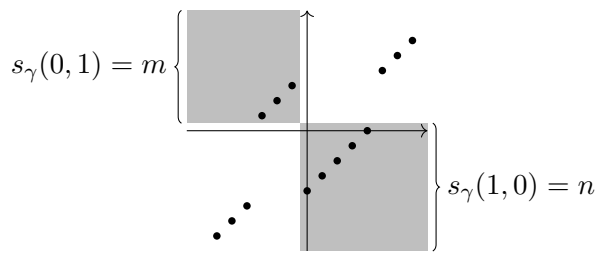

One additional benefit of bounded-difference permutations is that we can state a slightly simplified criterion for checking that a given function is the slipface function of a bounded-difference permutation (or equivalently, a submodular Clifford slipface). 

\begin{prop}
\label{prop:alphaExists}
A function $s: \ZZ^2 \to \ZZ$ is the slipface of a bounded-difference permutation (equivalently, a submodular Clifford slipface) if and only if the following two criteria hold.
\begin{enumerate}
\item There exist integers $M, \chi$ such that $a - b \leq -M$ implies $s(a,b) = 0$ and $a-b \geq M$ implies $s(a,b) = \chi + a - b$.
\item For all $a,b \in \ZZ$, $s(a+1,b) - s(a,b) - s(a+1,b+1) + s(a,b+1) \geq 0$ ($s$ is \emph{submodular}).
\end{enumerate}
The shift of $\alpha$ is the number $\chi$ mentioned in criterion (1).
\end{prop}

\begin{proof}
If $s$ is a submodular Clifford slipface, then (2) holds by definition and the boundedness of $M''_\alpha$ and $M''_{\alpha^{-1}}$ implies (1). Conversely, suppose that $s$ is a function satisfying these two criteria.
We first check that $s$ is a slipface using Lemma \ref{lem:dualCrit}. 
Submodularity implies that the first finite differences $s(a+1,b) - s(a,b)$ and $s(a,b) - s(a,b+1)$ are nondecreasing in $a$ and nonincreasing in $b$. Criterion $(1)$ implies that both differences are $0$ for $a-b$ sufficiently small, so both differences are nonnegative for all $a,b$, and achieve $0$ in each row and column. This implies (D1) and (D2) for $s$. The same argument applies to $s^\vee$. So $s$ is a submodular slipface, hence $s = \sa$ for a unique permutation $\alpha \in \asp$. Proposition \ref{prop:cliffordPerms} implies that $\alpha$ has bounded difference.
\end{proof}

\section{\texorpdfstring{Some subgroups of $\asp$ closed under $\star$ and $\resL$}{Some subgroups of asp closed under Demazure product and residual}}
\label{sec:subgroups}

We conclude this paper by examining a few types of commonly-encountered subgroups of $\asp$ that are closed under $\star$ and $\resL$ (and therefore also $\resR$), and pointing out some notable special features on the operations on these groups. Several of these examples have the following feature, which is a convenient way to verify closure under $\star$ and $\resL$.

\begin{defn}
Call a subgroup $G \leq \asp$ \emph{downward-closed} if for all $\alpha, \beta \in \asp$, if $\alpha \lechi \beta$ and $\beta \in G$, then $\alpha \in G$ as well.
\end{defn}

\begin{lemma}
\label{lem:dClosed}
If $G$ is downward-closed, then it is closed under $\star$ and $\resL$.
\end{lemma}

\begin{proof}
Suppose $\alpha, \beta \in G$. 
The ``greedy'' Theorem \ref{thm:starGreedy} implies that $\alpha \star \beta = \alpha_1 \beta$ for some $\alpha_1 \lechi \alpha$. Since $G$ is downward-closed, $\alpha_1 \in G$, hence so is $\alpha \star \beta$. The ``stingy'' Theorem \ref{thm:resLStingy} implies $\alpha \resL \beta = \alpha \beta_1$ where $\beta_1 \lechi \beta$. Then $\beta_1 \in G$, so $\alpha \resL \beta \in G$ as well.
\end{proof} 

\subsection{Permutations of bounded difference}
As in Section \ref{sec:ess}, we say that $\alpha$ has \emph{bounded difference} if there exists an integer $M$ such that $|\alpha(n) - n| \leq M$ for all $n \in \ZZ$. These permutations naturally occur in our application to algebraic curves and graphs (see Appendix \ref{app:curvesGraphs}), where the bound $M$ is determined by the genus, and we have seen that they have the virtue of being more easily compared in Bruhat order by using the essential set.

A permutation has difference bounded by $M$ if and only if $\sa(a-M,a) = \sai(a-M,a) = 0$ for all $a \in \ZZ$. Denote the group of such permutations by $B_M$. Then using again the fact that, if $\ca = \cb$, then $\alpha \leq \beta$ is equivalent to $\alpha^{-1} \leq \beta^{-1}$, it follows that $B_M$ is closed downwards. Thus so is the union of all the $B_M$. So the group of permutations of bounded difference is closed under $\star$ and $\resL$.

In fact, we can say a bit more: if $\alpha \in B_{M_1}$ and $\beta \in B_{M_2}$, then the permutations $\alpha_1, \beta_1$ in the proof of Lemma \ref{lem:dClosed} lie in $B_{M_1}$ and $B_{M_2}$, respectively. It follows that $\alpha \star \beta = \alpha_1 \beta$ and $\alpha \resL \beta = \alpha \beta_1$ both must lie in $B_{M_1+M_2}$. So the bound on the difference cannot grow too quickly under $\star$ and $\resL$.

\subsection{Symmetric groups}

To study symmetric groups, it is useful to first examine somewhat larger groups, of which symmetric groups may be formed as intersections.

\begin{defn}
Let $F_M \subseteq \asp$ consist of all permutations that send $\{ n \in \ZZ: n \geq M \}$ bijectively to itself (and thus also send $\{n \in \ZZ: n < M\}$ bijectively to itself).
\end{defn}

Equivalently, $\alpha \in F_M$ if and only if $\sa(M,M) = \sai(M,M) = 0$. This implies 

\begin{lemma}
For any $M \in \ZZ$, $F_M$ is downward-closed, hence closed under $\star$ and $\resL$.
\end{lemma}

Now, $S_d = \bigcap_{M \leq 1} F_M \cap \bigcap_{M \geq d+1} F_M$ and therefore

\begin{cor}
The symmetric group $S_d$ is downward-closed, and thus closed under $\star$ and $\resL$.
\end{cor}

As promised in the introduction, we can conveniently restrict arguments to $\sa$ to the set $[1,d+1]$ when working in $S_d$.

\begin{cor}
\label{cor:SdMin}
If $\alpha, \beta \in S_d$, then for all $1 \leq a,b \leq d+1$,
$$\sab(a,b) = \min_{1 \leq \ell \leq d+1} \Big[ \sa(a,\ell) + \sbe(\ell,b) \Big].$$
\end{cor}
\begin{proof}
Define $L$ as in Lemma \ref{lem:setL}, and let $\ell$ be any element of $L$. Since $\beta$, and thus $\beta^{-1}$, fixes all integers greater than $d$, $\beta^{-1}(\ell-1) < b \leq d+1$ implies $\ell -1 \leq d$. Similarly, since $\beta^{-1}$ fixes all nonpositive integers, $\beta^{-1}(\ell) \geq b \geq 1$ implies $\ell \geq 1$. Hence $L \subseteq \{1,\cdots, d+1\}$, and the corollary follows from Lemma \ref{lem:setL}.
\end{proof}

Of course, one often want to work with elements of $S_d$ via their reduced words, i.e. factorization into simple transpositions (or ``bubble-sorts''). We can formulate the necessary tools here somewhat generally, so we do so for later use. In the following statements, recall the notation $\sigma_S$ from Definition \ref{defn:simpleInv} for a product of simple transpositions.

\begin{thm}
\label{thm:alphaStarSigma}
Let $\alpha \in \asp$, and let $S$ be a set with no two consecutive integers. Define $S_1 = \{n \in S: \alpha(n) < \alpha(n+1) \}$ and $S_2 = \{n \in S: \alpha(n) > \alpha(n+1) \}$. Then
\begin{flalign*}
&& \alpha \star \sigma_S &= \alpha \sigma_{S_1}, \mbox{ and } & \leanlink{Transpositions.starSigma}\\
&& \alpha \resL \sigma_S &= \alpha \sigma_{S_2}. & \leanlink{Transpositions.residualSigma}
\end{flalign*}
\end{thm} 

\begin{proof}
\newcommand{\stwo}{\sigma_{S_2}}
By definition, $\sigma_S = \sigma_{S_1} \sigma_{S_2} = \sigma_{S_2} \sigma_{S_1}$. Since $S_1, S_2$ are disjoint, both these products are reduced.
Corollary \ref{cor:sasss} implies that $\alpha \star \sigma_{S_2} = \alpha$ and $\alpha \resL \sigma_{S_1} = \alpha$.
Lemma \ref{lem:reducedStar} implies that $\alpha \starr \sigma_{S_1} = \alpha \sigma_{S_1}$ and Lemma \ref{lem:reducedRes} implies that $\alpha \resLr \sigma_{S_2} = \alpha \sigma_{S_2}$. Putting this together and using associativity and Lemma \ref{lem:sfAlgebra},
$\alpha \star \sigma_S = (\alpha \star \sigma_{S_2}) \star \sigma_{S_1} = \alpha \star \sigma_{S_1} = \alpha \sigma_{S_1}$, and $\alpha \resL \sigma_S = (\alpha \resL \sigma_{S_1}) \resL \sigma_{S_2} = \alpha \resL \sigma_{S_2} = \alpha \sigma_{S_2}$, as desired.
\end{proof}

So Demazure products and residuals may be conveniently computed by factoring one of the permutations into adjacent transpositions, and this can also be slightly ``parallelized'' by considering many non-overlapping adjacent transpositions at once.
Theorem \ref{thm:alphaStarSigma} implies the last sentences of Theorems \ref{thm:starExists} and \ref{thm:resL} from the introduction, by taking $S = \{n\}$.

\subsection{Permutations with finitely many inversions}
Let $G$ denote the group of permutations $\alpha$ such that $\Inv(\alpha)$ is finite. Alternatively, $G$ is the semidirect product $S_\ZZ \rtimes \ZZ$ of the symmetric group of $S_\ZZ$ of permutations fixing all but finitely many integers with the group $\ZZ \cong \{\iota_\chi: \chi \in \ZZ\}$ of shift permutations.
Then $G$ is closed under $\resL$ in a strong sense: for any $\alpha \in G$ and $\beta \in \asp$ (not necessarily in $G$), $\alpha \resL \beta \in G$. This is because $\alpha \resL \beta \leq_R \alpha$ (Lemma \ref{lem:invRes}), so $\Inv ((\alpha \resL \beta)^{-1}) \subseteq \Inv (\alpha^{-1})$ is finite. 

This group is also closed downward, and therefore closed under $\star$. This can be seen from a quick induction argument. If $\alpha \lechi \beta$, and $\Inv (\beta) \neq \emptyset$, then there exists some adjacent inversion $(n,n+1) \in \Inv (\beta)$. Then $\alpha \resL \sigma_n \leq \beta \resL \sigma_n$. Theorem \ref{thm:alphaStarSigma} implies that $\beta \resL \sigma_n = \beta \sigma_n$ has one fewer inversion, and $\alpha \resL \sigma_n$ is either $\alpha$ or $\alpha \sigma_n$, which either both have finitely many inversions or both do not. In the base case, if $\beta$ has no inversions, then $\beta = \iota_{\cb}$, so $\alpha = \iota_{\cb}$ as well.

In fact, with just a bit more care the inductive argument above shows the \emph{subword property} \cite[Theorem 2.2.2]{bjornerBrenti}: if $\beta = \iota_\chi \sigma_{n_1} \cdots \sigma_{n_\ell}$, where $\ell = \# \Inv(\beta)$, and $\alpha \lechi \beta$, then there exists a subsequence $1 \leq i_1 < \cdots < i_m \leq \ell$ such that $\alpha = \iota_\chi \alpha_{i_1} \cdots \alpha_{i_m}$ and $m = \# \Inv(\alpha)$. In particular, $\# \Inv(\alpha) \leq \# \Inv (\beta)$. This fact provides quick proofs of the following two inequalities. If both $\alpha,\beta$ have finitely many inversions, then
$$\# \Inv (\alpha \star \beta) \leq \# \Inv(\alpha) + \# \Inv(\beta)
\mbox{ and }
\# \Inv (\alpha \resL \beta) \geq \# \Inv(\alpha) - \# \Inv(\beta).
$$
For the first, write $\alpha \star \beta = \alpha_1 \beta$, with $\alpha_1 \leq \alpha$ (Theorem \ref{thm:starGreedy}). Equality holds if and only if $\alpha \starr \beta$. For the second, write $\alpha \resL \beta = \alpha \beta_1$ with $\beta_1 \leq \beta$ (Theorem \ref{thm:resLStingy}). Equality holds if and only if $\beta^{-1} \leq_L \alpha$.

The stronger form of closure under $\resL$ allows the following strong form of the reduction theorem. This follows from the second paragraph of Theorem \ref{thm:reduce}.

\begin{prop}
Suppose that $\alpha \star \beta \geq \gamma$, where $\gamma$ (but not necessarily $\alpha$ or $\beta$) has finitely many inversions. Then there exist $\alpha_1 \lechi \alpha, \beta_1 \lechi \beta$, with finitely many inversions, such that $\alpha_1 \starr \beta_1 = \gamma$.
\end{prop}

Observe also that permutations with finitely many inversions must also have bounded difference and finite essential set.

\subsection{The (extended) affine symmetric groups}
\label{ssec:extAffine}

Fix an integer $k \geq 2$. The \emph{extended affine symmetric group} of modulus $k$ is the group of permutations $\alpha$ such that $\alpha(n+k) = \alpha(n) + k$ for all $n \in \ZZ$. This condition is equivalent to $\iota_k \alpha \iota_{-k} = \alpha$, and therefore is equivalent to the condition 
$$\sa(a,b) = \sa(a+k,b+k) \mbox{ for all } a,b \in \ZZ.$$
See for example \cite[\S 8.3]{bjornerBrenti} for background on affine symmetry groups. In particular, see \cite[Theorem 8.3.7]{bjornerBrenti} for discussion of functions like the slipfaces $\sa$ of the present paper in the context of these groups.

It is immediate from the definitions of $\star$ and $\resL$ that if $\sa, \sbe$ satisfy this condition, then so does $\sab$. So the extended affine symmetric group is closed under $\star$ and $\resL$. As mentioned in Example \ref{eg:extAffine}, the \emph{affine symmetric group} $\ts_k$ is the subgroup of the shift-$0$ permutations in the extended affine symmetric group. Since shift-$0$ permutations are closed under $\star$ and $\resL$, so is $\ts_k$.

The affine symmetric group $\ts_k$ is generated by the permutations $\widetilde{\sigma}_n = \sigma_{n + k \ZZ}$ (one quick way to see this is to define the length of $\alpha \in \ts_k$ to be the number of inversions $(m,n)$ with $0 \leq m < k$, and show that this is finite and, if positive, can be reduced by multiplying by some $\widetilde{\sigma}_n$). Therefore the operations $\star$ and $\resL$ can be easily computed via factorization into the elements $\widetilde{\sigma}_n$, using Theorem \ref{thm:alphaStarSigma}. If $S = n + k \ZZ$ in the theorem statement, and $\alpha \in \ts_k$, then one of $S_1, S_2$ is empty, so one of $\alpha \star \widetilde{\sigma}_n, \alpha \resL \widetilde{\sigma}_n$ is $\alpha \widetilde{\sigma}_n$, and the other is $\alpha$, according to whether or not $\alpha(n) < \alpha(n+1)$.

Since $\ts_k$ is a Coxeter group, this description of $\star$ and $\resL$ is known, so there is nothing novel here except verifying that the standard definition of these operations in a Coxeter group accords with the definition on $\asp$ via Equation \eqref{eq:sab}.

\appendix
\section{\texorpdfstring{A geometric interpretation of the min-plus formula for $\star$}{A geometric interpretation of the min-plus formula for star}}
\label{app:geometric}

The Demazure product on $S_d$ arises naturally in the geometry of flag varieties and Schubert calculus; we briefly summarize a situation that provides useful intuition for Equation \eqref{eq:sab}. This discussion parallels the discussion in \cite[\S 2.3]{cpRR}, although the notation is different.
 
Elements of $S_d$ can be used to measure the ``distance'' between two complete flags in an $n$-dimensional vector space; the identity permutation corresponds to identical flags, and the descending permutation to transverse flags. Let $H$ be a $d$-dimensional vector space. Let $U_\sbu = (U_a)_{0 \leq a \leq d}, V_\sbu = (V_b)_{0 \leq b \leq d}$ be two complete flags in $H$, indexed by dimension. 
There exists a unique permutation $\sigma = \sigma(U_\sbu, V_\sbu)$ such that for all $0 \leq a,b \leq d$,
 \begin{equation}
 \label{eq:suv}
 \dim U_a \cap V_b = \# \{ n \leq b: \sigma(n) \leq a \} = a - s_\sigma(a+1,b+1).
 \end{equation}
 Note that this definition of $\sigma(U_\sbu, V_\sbu)$ differs from the one used in \cite{cpRR} and some other sources; we use it here to fit better the notation of this paper.
 Another description of $\sigma$ is via \emph{adapted bases.} A basis $\cB$ is adapted to a flag $U_\sbu$ if $\cB \cap U_a$ is a basis for $U_a$, for all $0 \leq a \leq d$. It is always possible to find a basis $\cB$ that is adapted to both of two given flags $U_\sbu, V_\sbu$; this can be proved by an argument in Gaussian elimination. If this basis is ordered $\{u_1, \cdots, u_d\}$ such that $\{u_1, \cdots, u_b\}$ is a basis of $V_b$ for all $b$, then there is a unique permutation $\sigma$ of $\{1, \cdots, d\}$ such that for all $a$, $\{ u_{\sigma(1)}, \cdots, u_{\sigma(a)} \}$ is a basis for $U_a$. Then $U_a \cap V_b$ has basis $\{ u_i: i \leq b \mbox{ and } \sigma(i) \leq a\}$, and Equation \eqref{eq:suv} follows.
 
For example, if $U_\sbu = V_\sbu$, then 
$\dim U_a \cap V_b = \min(a,b)$, so $$s_\sigma(a,b) = a-1-\min\{a-1,b-1\} = \max \{0,a-b\},$$
 and $\sigma(U_\sbu, V_\sbu)$ is the identity permutation. At the other extreme, if the flags are transverse then $\dim U_a \cap V_b = \max\{0,a + b - d\}$ and $\sigma(U_\sbu, V_\sbu)$ is the descending permutation.
 
The Demazure product answers the question: given \emph{three} flags $U_\sbu, V_\sbu, W_\sbu$ and two permutations $\alpha, \beta$, if $\sigma(U_\sbu, V_\sbu) = \alpha$ and $\sigma(V_\sbu,W_\sbu) = \beta$, then what is the Bruhat-maximum of all possible $\sigma(U_\sbu, W_\sbu)$?

To see why the Demazure product answers this question, note that for all $a,b,\ell \in \{0,\cdots, d\}$,
$$\dim U_a \cap W_b \geq \dim U_a \cap V_\ell \cap W_b  \geq \dim U_a \cap V_\ell + \dim V_\ell \cap W_b - \dim V_\ell.$$
In the notation of slipface functions, this may be rewritten, after subtracting both sides from $a$, as
$$s_{\sigma(U_\sbu, W_\sbu)}(a+1,b+1) \leq s_{\sigma(U_\sbu, V_\sbu)}(a+1,\ell+1) + s_{\sigma(V_\sbu, W_\sbu)}(\ell+1,b+1).$$

Taking the maximum over $\ell \in \{0, \cdots, d\}$, Corollary \ref{cor:SdMin} gives
$$\sigma(U_\sbu, W_\sbu) \leq \sigma(U_\sbu, V_\sbu) \star \sigma(V_\sbu, W_\sbu).$$
In fact, for ``generic'' choices of flags, equality is obtained (see e.g. \cite[Theorem 1.1]{cpRR}).

The discussion in this section can be generalized, with some care, from $S_d$ to $\asp$ if one considers infinite flags, indexed by $\ZZ$, in an infinite-dimensional vector space. Rather than expressing the dimension of $U_a \cap V_b$ as $a - s_\sigma(a+1,b+1)$, we may write $\operatorname{rank} \left( U_a \to H / V_b \right) = s_{\sigma}(a+1,b+1)$, which is applicable in this infinite-dimensional setting provided that these ranks are all finite.

\section{Applications to curves and graphs}
\label{app:curvesGraphs}
\newcommand{\st}{s_\tau}
This paper originated in service of applications to Brill-Noether theory of curves and graphs.
We summarize the basic ideas without full details in this appendix to motivate the content of this paper; for more details the reader may consult the references mentioned.
This appendix assumes vocabulary from algebraic curves and the divisor theory of metric graphs.
The basic goal in both settings is to provide a toolkit for inductive arguments, in which the geometry of a curve or graph with two marked points is reduced to the geometry of \emph{two} such curves/graphs that are glued at a single point.

Applications to algebraic curves are developed in \cite{hbnDemazure}. Let $C$ be a smooth projective algebraic curve over an algebraically closed field, and let $p,q$ be two distinct points on $C$. Then a line bundle $\cL$ on $C$ determines a permutation $\tau = \tau_\cL^{p,q}$ by
\begin{eqnarray}
\label{eq:tauL}
h^0(\cL(ap-bq)) &=& \# \{ n \geq b:\ \tau(n) \leq a \}\\
&=& \st(a+1,b).
\end{eqnarray}

The existence of $\tau$ follows from Proposition \ref{prop:alphaExists}, which also proves that $\tau$ has bounded difference. The submodularity of $s_\tau$ amounts to the following linear algebraic observation.
$$h^0(\cL) - h^0(\cL(-p)) - h^0(\cL(-q)) + h^0(\cL(-p-q)) = \dim H^0(\cL) / \left[ H^0(\cL(-p)) + H^0(\cL(-q)) \right] \geq 0$$
Call $\tau^{p,q}_\cL$ the \emph{transmission permutation} of $\cL$ with respect to $p,q$. 
The shift of $\tau^{p,q}_\cL$ is given by the Euler characteristic as $\chi_\tau = \chi(\cL)-1 = d-g$, where $d = \deg \cL$ and $g$ is the genus of $C$.
Reversing this process, we may define for any $\tau \in \asp$ a \emph{transmission locus}
$$W^\tau(C,p,q) = \{ [\cL] \in \Pic^{\chi_\tau + g}(C):\ \tau^{p,q}_\cL \geq \tau \}.$$
In particular, for any choice of $r,d$, we may consider the permutation $\gamma = \gamma^{g-d+r}_{r+1}$ from Example \ref{eg:bnPerm}, for which $\chi_\gamma = d-g$ and $\ess(\gamma) = \{(1,0)\}$. Therefore
\begin{eqnarray}
W^\gamma(C,p,q) &=& W^r_d(C)\\
&:=&  \{ [\cL] \in \Pic^d(C):\ h^0(C,\cL) \geq r+1 \}.
\end{eqnarray}
This locus is called a \emph{Brill--Noether locus}. These loci play a central role in the theory of algebraic curves. So these transmission loci provide a reframing, and generalization, of Brill--Noether loci. 

In another direction, there has been substantial recent interest in \emph{Hurwitz--Brill--Noether theory}, which concerns the classification of line bundles on a general $k$-gonal curve of genus $g$. See for example \cite{larsonLarsonVogt} and the references therein. The principal objects of study are \emph{splitting loci} $W^{\vec{e}}(C)$, for $k$-gonal curves $C$. Transmission loci are a useful tool for studying these loci in a special case: if $(C,p,q)$ is a $k$-gonal curve with $p,q$ points of total ramification for the degree-$k$ cover $\pi: C \to \PP^1$, then for every splitting type $\vec{e}$ there exists an \emph{extended $k$-affine permutation} $\gamma_{\vec{e}}$ such that
$$W^{\gamma_{\vec{e}}}(C,p,q) = W^{\vec{e}}(C).$$

The utility of this reframing is that transmission loci are very well-suited to induction arguments, once the Demazure product is in hand. 
The construction of transmission permutations and transmission loci can be generalized in a natural way to chains of smooth curves. The Demazure product provides a ``gluing'' operation. Precisely, if $(C_1, p_1, q_1), (C_2,p_2,q_2)$ are two twice-marked smooth curves, $X~=~C_1~\cup~C_2$ is the nodal curve obtained by gluing $q_1$ to $p_2$, and $\cL$ is a line bundle on $X$ restricting to $\cL_1$ on $C_1$ and $\cL_2$ on $C_2$, then
\begin{equation}
\label{eq:glueCurves}
\tau_\cL^{p_1, q_2} = \tau_{\cL_1}^{p_1, q_1} \star \tau_{\cL_2}^{p_2,q_2}.
\end{equation}
\begin{rem}
The apparently strange definition of $\tau^{p,q}_\cL$, in which $\st(a+1,b)$ appears rather than $\st(a,b)$, was chosen so that \eqref{eq:glueCurves} would be as simple and intuitive as possible.
\end{rem}
In other words, the transmission locus on a union of two curves (or chains of curves) decomposes:
\begin{equation}
\label{eq:WtauDecomp}
W^\tau(C, p_1, q_2) = \bigcup_{\alpha \star \beta \geq \tau} W^\alpha(C_1, p_1, q_1) \times W^\beta(C_2, p_2, q_2).
\end{equation}
This union is indexed over quite a large set: all pairs of permutations $\alpha, \beta \in \asp$ with these shifts such that $\alpha \star \beta \geq \tau$. The utility of the reduction Theorem \ref{thm:reduce} is to shrink this to a \emph{finite} union of pieces that all have the same expected dimension. Namely, 
\begin{equation}
\label{eq:WtauRedDecomp}
W^\tau(C, p_1, q_2) = \bigcup_{\alpha \starr \beta = \tau} W^\alpha(C_1, p_1, q_1) \times W^\beta(C_2, p_2, q_2).
\end{equation}
We have omitted the constraints on $\ca,\cb$ from the notation for simplicity. Equation \eqref{eq:WtauRedDecomp} now provides a convenient inductive tool for analyzing both
\begin{enumerate}
\item local geometry of transmission loci, and thereby Brill--Noether loci, and
\item enumerative aspects of transmission loci, and thereby Brill--Noether loci.
\end{enumerate}
In particular, \cite{hbnDemazure} uses this perspective to obtain a new proof of the classical Brill--Noether theorem and the main dimension upper bound of Hurwitz--Brill--Noether theory within a single framework.

A completely parallel story may be developed for finite graphs, and also for metric graphs. Details and applications can be found in \cite{pflBriefTropBN,pflSolomon}. Let $G$ be either a graph or metric graph, and $p,q$ two distinct vertices on $G$. If $D$ is a divisor on $G$, then we may associate a slipface function to $D$, by
\begin{equation}
\label{eq:tauD}
s_D^{p,q}(a,b) = r(D + (a+1)p - bq) + 1.
\end{equation}
Here $r$ is the Baker-Norine rank. 
As in the curve case, the shift of this slipface is $d-g$.
Unfortunately, this slipface need not be submodular (though it is always Clifford), so there may not be a well-defined ``transmission permutation.'' 
The problem, compared to algebraic curves, is that the Baker-Norine rank is not the dimension of a vector space. However, it appears that in many cases, such as chains of loops, all divisors do have submodular slipfaces.

In either case, this slipface still satisfies an easy-to-remember gluing equation. Let $G_1$ be a (metric) graph with two points $p_1, q_1$, and $G_2$ a (metric) graph with two points $p_2, q_2$. Let $G$ be obtained by gluing $q_1$ to $p_2$. Then for any pair of divisors $D_1$ on $G_1$, $D_2$ on $G_2$, if we let $D$ denote the sum of these two divisors on $G$, then the analog of Equation \eqref{eq:glueCurves} holds for (metric) graphs:
\begin{equation}
\label{eq:glueGraphs}
s_D^{p_1, q_2} = s_{D_1}^{p_1,q_1} \star s_{D_2}^{p_2, q_2}.
\end{equation} 
This framework now provides a tool for the Brill--Noether theory of chains of graphs, in the same fashion as for chains of curves.

The paper \cite{pflBriefTropBN} studies this situation in a simple case: chains of twice marked loops, all of the same torsion order. This situation is somewhat simple, in that all transmission permutations involved belong to the same extended affine symmetric group. The stronger results in this paper make somewhat more subtle arguments possible, in which different torsion orders are used in a chain, and therefore it is necessary to intermingle extended affine permutations of different moduli. This is carried out, for example, in \cite{pflSolomon} to construct new families of Brill--Noether general graphs.

\section*{Acknowledgements}

This work was supported by a Miner D. Crary Sabbatical Fellowship from Amherst College, and the author was hosted by the Max Planck Institute for Mathematics in the Sciences during revisions and the Lean formalization process.
I am grateful to Will Erickson, Darij Grinberg, Allen Knutson, Thomas Lam, Leonardo Mihalcea, Travis Scrimshaw, and Mark Shimozono for references and suggestions, and to Alex Tiskin for directing me to his original proof of the min-plus formula for the Demazure product on symmetric groups.


\bibliographystyle{amsalpha}
\bibliography{demazure.bbl}
\end{document}